\documentclass{amsart}
\usepackage{amscd}
\usepackage{amssymb}
\usepackage[T1]{fontenc}
\usepackage[utf8]{inputenc}
\usepackage{hyperref}
\usepackage[arrow,curve,matrix]{xy}
\usepackage{times}

\DeclareFontFamily{OMS}{rsfs}{\skewchar\font'60}
\DeclareFontShape{OMS}{rsfs}{m}{n}{<-5>rsfs5 <5-7>rsfs7 <7->rsfs10 }{}
\DeclareSymbolFont{rsfs}{OMS}{rsfs}{m}{n}
\DeclareSymbolFontAlphabet{\scr}{rsfs}

\renewcommand{\div}{\rm div}
\renewcommand{\P}{\mathbb{P}}
 
\renewcommand{\O}{\sO}%{\mathcal{O}}
\newcommand{\Q}{\mathbb{Q}}

\newcommand{\sA}{\scr{A}}
\newcommand{\sB}{\scr{B}}
\newcommand{\sC}{\scr{C}}

\newcommand{\sF}{\scr{F}}
\newcommand{\sG}{\scr{G}}

\newcommand{\sO}{\scr{O}}

\newcommand{\bN}{\mathbb{N}}

\newcommand{\bQ}{\mathbb{Q}}

\newcommand{\cC}{\mathcal{C}}

\newcommand{\ga}{\gamma}

\newcommand{\la}{\lambda}

\newcommand{\GA}{\Gamma}

\newcommand{\OM}{\Omega}

\DeclareMathOperator{\Cl}{Cl}
\DeclareMathOperator{\coker}{coker}
\DeclareMathOperator{\Div}{Div}
\DeclareMathOperator{\discrep}{discrep}

\DeclareMathOperator{\Chow}{Chow}
\DeclareMathOperator{\codim}{codim}

\DeclareMathOperator{\hol}{hol}

\DeclareMathOperator{\Image}{Image}

\DeclareMathOperator{\Pic}{Pic}

\DeclareMathOperator{\rank}{rank}
\DeclareMathOperator{\red}{red}
\DeclareMathOperator{\reg}{reg}
\DeclareMathOperator{\sml}{small}

\DeclareMathOperator{\Sym}{Sym}
\DeclareMathOperator{\supp}{supp}
\DeclareMathOperator{\tor}{tor}
\DeclareMathOperator{\Var}{Var}

\newcommand{\into}{\hookrightarrow}

\newcommand{\iso}{\simeq}

\def\coh#1.#2.#3.{H^{#1}(#2,#3)}
\def\cohd#1.#2.#3.{H_{dR}^{#1}(#2,#3)}
\def\dimcoh#1.#2.#3.{h^{#1}(#2,#3)}
\def\loccoh#1.#2.#3.#4.{H^{#1}_{#2}(#3,#4)}
\def\loccohmod#1.#2.#3.{H^{#1}_{#2}(#3)}
\def\dimloccoh#1.#2.#3.#4.{h^{#1}_{#2}(#3,#4)}
\def\ses#1.#2.#3.{0  \longrightarrow  #1   \longrightarrow 
 #2 \longrightarrow #3 \longrightarrow 0} 
\def\sesshort#1.#2.#3.{0
 \rightarrow #1 \rightarrow #2 \rightarrow #3 \rightarrow 0}
\newcounter{thisthm}

\newcommand{\ilabel}[1]{\newcounter{#1}\setcounter{thisthm}{\value{thm}}\setcounter{#1}{\value{enumi}}}
\newcommand{\iref}[1]{(\thesection.\the\value{thisthm}.\the\value{#1})}

\theoremstyle{plain}    
\newtheorem{thm}{Theorem}[section]
\newtheorem{defn}[thm]{Definition}

\newtheorem{assumption}[thm]{Assumption} 
\newtheorem{setup}[thm]{Setup} 
\numberwithin{equation}{thm}
\numberwithin{figure}{section}
\theoremstyle{plain}    
\newtheorem{cor}[thm]{Corollary}
\newtheorem{cordefn}[thm]{Corollary and Definition}
\newtheorem{lem}[thm]{Lemma}
\newtheorem{conjecture}[thm]{Conjecture}

\theoremstyle{plain}    
\newtheorem{prop}[thm]{Proposition}
\newtheorem{proclaim-special}[thm]{\specialthmname}

\theoremstyle{remark}
\newtheorem{rem}[thm]{Remark}

\newtheorem{explanation}[thm]{Explanation}
\newtheorem{computation}[thm]{Computation}
\newtheorem{obs}[thm]{Observation}
\newtheorem{warning}[thm]{Warning}
\newtheorem{subrem}[equation]{Remark}
\newtheorem{subclaim}[equation]{Claim} %%Delete [...] to re-start numbering
\newtheorem*{claim*}{Claim} 
\newtheorem{notation}[thm]{Notation}
\newtheorem{convention}[thm]{Convention}

\newtheorem{construction-defn}[thm]{Construction and Definition}

\newtheoremstyle{bozont-remark}{3pt}{3pt}%
     {}%         Body font
     {}%         Indent amount (empty = no indent, \parindent = para indent)
     {\it}% Thm head font
     {.}%        Punctuation after thm head
     {.5em}%     Space after thm head (\newline = linebreak)
     {\thmname{#1}\thmnumber{ #2}: \thmnote{\sc #3}}%         Thm head spec
\theoremstyle{bozont-remark}

\def\factor#1.#2.{\left. \raise 2pt\hbox{$#1$} \right/\hskip -2pt\raise
  -2pt\hbox{$#2$}} 
\newlength{\swidth}
\newenvironment{enumerate-p}{
  \begin{enumerate}}
  {\setcounter{equation}{\value{enumi}}\end{enumerate}}

\date{\today}

\author{Kelly Jabbusch}
\author{Stefan Kebekus}

\thanks{Kelly Jabbusch and Stefan Kebekus were supported by the
  DFG-Forschergruppe ``Classification of Algebraic Surfaces and Compact
  Complex Manifolds'' in full and in part, respectively. The work on this
  paper was finished while the authors visited the 2009 Special Year in
  Algebraic Geometry at the Mathematical Sciences Research Institute,
  Berkeley. Both authors would like to thank the MSRI for support.}
  
 \address{Kelly Jabbusch, Department of Mathematics, KTH, 10044 Stockholm, Sweden}
\email{\href{mailto:jabbusch@math.kth.se}{jabbusch@math.kth.se}}
 \urladdr{\href{http://www.math.kth.se/~jabbusch}{http://www.math.kth.se/~jabbusch}} 
  
\address{Stefan Kebekus, Mathematisches Institut, Albert-Ludwigs-Universit\"at
  Freiburg, Eckerstra{\ss}e 1, 79104 Freiburg im Breisgau, Germany}
\email{\href{mailto:stefan.kebekus@math.uni-freiburg.de}{stefan.kebekus@math.uni-freiburg.de}}
\urladdr{\href{http://home.mathematik.uni-freiburg.de/kebekus}{http://home.mathematik.uni-freiburg.de/kebekus}}

\title{Families over special base manifolds and a conjecture of Campana}

\date{\today}
\begin{document}

\begin{abstract}
  Consider a smooth, projective family of canonically polarized varieties over
  a smooth, quasi-projective base manifold $Y$, all defined over the complex
  numbers. It has been conjectured that the family is necessarily isotrivial
  if $Y$ is special in the sense of Campana. We prove the conjecture when $Y$
  is a surface or threefold.
  
  The proof uses sheaves of symmetric differentials associated to fractional
  boundary divisors on log canonical spaces, as introduced by Campana in his
  theory of \emph{Orbifoldes G\'eom\'etriques}. We discuss a weak variant of
  the Harder-Narasimhan Filtration and prove a version of the
  Bogomolov-Sommese Vanishing Theorem that take the additional fractional
  positivity along the boundary into account. A brief, but self-contained
  introduction to Campana's theory is included for the reader's convenience.
\end{abstract}

\maketitle
\tableofcontents

\section{Introduction and statement of main result}
\label{sec:intro}

\subsection{Introduction}

Complex varieties are traditionally classified by their Kodaira-Iitaka
dimension. A smooth, projective variety $Y$ is said to be of ``general type''
if the Kodaira-Iitaka dimension of the canonical bundle is maximal,
i.e.~$\kappa(\Omega^{\dim Y}_Y) = \dim Y$. Refining the distinction between
``general type'' and ``other,'' Campana suggested in a series of remarkable
papers to consider the class of ``special'' varieties $Y$, characterized by
the fact that the Kodaira-Iitaka dimension $\kappa(\sA)$ is small whenever
$\sA$ is an invertible subsheaf of $\Omega^p_Y$, for some $p$. Replacing
$\Omega^p_Y$ with the sheaf of logarithmic differentials, the notion also
makes sense for quasi-projective varieties.

Conjecturally, special varieties have a number of good topological,
geometrical and arithmetic properties. In particular, Campana conjectured
that any map from a special quasi-projective variety to the moduli stack of
canonically polarized manifolds is necessarily constant. Equivalently, it is
conjectured that any smooth projective family of canonically polarized
manifolds over a special quasi-projective base variety is necessarily
isotrivial. This generalizes the classical Shafarevich Hyperbolicity Theorem
and recent results obtained for families over base manifolds that are not of
general type, cf.~\cite{KK08, KK08c} and the references therein.

In this paper, we prove Campana's conjecture for quasi-projective base
manifolds $Y^\circ$ of dimension $\dim Y^\circ \leq 3$.  Throughout the
present paper we work over the field of complex numbers.

\subsection{Main result}

Before formulating the main result in Theorem~\ref{thm:main} below, we briefly
recall the precise definition of a special logarithmic pair. The classical
Bogomolov-Sommese Vanishing Theorem is our starting point.

\begin{thm}[\protect{Bogomolov-Sommese Vanishing Theorem, \cite[Sect.~6]{EV92}}]\label{thm:classBSv}
  Let $Y$ be a smooth projective variety and $D \subset Y$ a reduced, possibly
  empty divisor with simple normal crossings. If $p \leq \dim Y$ is any number
  and $\sA \subseteq \Omega^p_Y(\log D)$ any invertible subsheaf, then the
  Kodaira-Iitaka dimension of $\sA$ is at most $p$, i.e., $\kappa(\sA) \leq
  p$. \qed
\end{thm}

In a nutshell, we say that a pair $(Y, D)$ is special if the inequality in the
Bogomolov-Sommese Vanishing Theorem is always strict.

\begin{defn}[Special logarithmic pair]\label{def:speciallog}
  In the setup of Theorem~\ref{thm:classBSv}, a pair $(Y, D)$ is called
  \emph{special} if the strict inequality $\kappa(\sA) < p$ holds for all $p$
  and all invertible sheaves $\sA \subseteq \Omega^p_Y(\log D)$. A smooth,
  quasi-projective variety $Y^\circ$ is called special if there exists a
  smooth compactification $Y$ such that $D:= Y \setminus Y^\circ$ is a divisor
  with simple normal crossings and such that the pair $(Y,D)$ is special.
\end{defn}

\begin{rem}[Special quasi-projective variety]
  It is an elementary fact that if $Y^\circ$ is a smooth, quasi-projective
  variety and $Y_1$, $Y_2$ two smooth compactifications such that $D_i := Y_i
  \setminus Y^\circ$ are divisors with simple normal crossings, then $(Y_1,
  D_1)$ is special if and only if $(Y_2, D_2)$ is. The notion of special
  should thus be seen as a property of the quasi-projective variety $Y^\circ$.
\end{rem}

With this notation in place, Campana's conjecture is then formulated as
follows.

\begin{conjecture}[\protect{Generalization of Shafarevich Hyperbolicity, \cite[Conj.~12.19]{Cam07}}]\label{conj:campana}
  Let $f ^\circ: X^\circ \to Y^\circ$ be a smooth family of canonically polarized
  varieties over a smooth quasi-projective base. If $Y^\circ$ is special, then
  the family $f^\circ$ is isotrivial.
\end{conjecture}

\begin{thm}[Campana's conjecture in dimension three]\label{thm:main}
  Conjecture~\ref{conj:campana} holds if $\dim Y^\circ \leq 3$.
\end{thm}

\begin{subrem}
  In the case of $\dim Y^\circ =2$, Conjecture~\ref{conj:campana} is claimed
  in \cite[Thm.~12.20]{Cam07}. However, we had difficulties following the
  proof, and offered a new proof of Campana's conjecture in dimension two,
  \cite[Cor.~4.5]{JK09}.
\end{subrem}

\begin{rem}
  In analogy to the maximally rationally connected fibration, Campana proves
  the existence of a quasi-holomorphic ``core map'', $c : Y^\circ \dasharrow
  C(Y^\circ)$, which is characterized by the fact that its fibers are special
  any by a certain maximality property. One equivalent reformulation of
  Conjecture~\ref{conj:campana} is that the core map always factors the moduli
  map $\mu : Y^\circ \to \mathfrak M$, i.e., that there exists a commutative
  diagram of rational maps
  $$
  \xymatrix{
    && C(Y^\circ) \ar@{-->}@/^0.2cm/[rrd] \\
    Y^\circ \ar[rrrr]_{\mu} \ar@{-->}@/^0.2cm/[rru]^{c} &&&& \mathfrak M.
  }
  $$
\end{rem}

\subsection{Outline of the paper}

In Part~\ref{P1} of this paper, we introduce the notion of $\cC$-pairs, also
called \emph{Orbifoldes G\'eom\'etriques} by Campana, and prove a number of
basic results that will be important later. The notion of a $\cC$-pair offers
the formal framework suitable for the discussion of differentials on charts of
moduli stacks and on the associated coarse moduli spaces.
Section~\ref{sec:orbifold} contains a brief introduction to $\cC$-pairs and
their use for our purposes.  Several sheaves of differentials and the
associated version of Kodaira-Iitaka dimension for subsheaves of
$\cC$-differentials are also introduced.

Even though our presentation differs from that of Campana's papers, most of
the material covered in Part~\ref{P1} is not new and appears, e.g., in
\cite{Cam07}. We have chosen to include a complete and entirely self-contained
introduction because we found some parts of \cite{Cam07} hard to read, and
because some of the basic notions have still not found their final form in the
literature.

In contrast, the results of Part~\ref{P2} are new to the best of our
knowledge. In Section~\ref{sec:slope}, we discuss a weak variant of the
Harder-Narasimhan Filtration that works for sheaves of $\cC$-differentials and
takes the extra fractional positivity of these sheaves into account. Even
though we believe that a refinement of the Harder-Narasimhan Filtration works
in the more general context of vector bundles with fractional elementary
transformations, and might be of independent interest, we develop the theory
only to the absolute minimum required to prove Theorem~\ref{thm:main}.

In Section~\ref{sec:BSv}, we generalize the classical Bogomolov-Sommese
Vanishing Theorem~\ref{thm:classBSv} to sheaves of $\cC$-differentials on
$\cC$-pairs with log canonical singularities. Again, this is a generalization
of the results obtained in \cite{GKK08} that respects the fractional
positivity along the boundary.

In Part~\ref{P3} we prove Theorem~\ref{thm:main}. To prepare for the proof we
recall in Section~\ref{sec:VZ} a recent refinement of Viehweg-Zuo's
fundamental positivity result: if the family $f^\circ$ is non-isotrivial and
if $Y$ is any smooth compactification of $Y^\circ$ such that $D:= Y \setminus
Y^\circ$ is a divisor with simple normal crossings, then there exists a number
$m \gg 0$ and an invertible subsheaf $\sA \subseteq \Sym^m \OM_Y^1(\log D)$ of
positive Kodaira-Iitaka dimension. In the appropriate orbifold sense, this
``Viehweg-Zuo sheaf'' $\sA$ comes from the moduli space. One of the main
difficulties in the proof of Theorem~\ref{thm:main} is that special pairs are
defined in terms of subsheaves in $\Omega^1_Y(\log D)$, while Viehweg-Zuo's
result only gives subsheaves of high symmetric products $\Sym^m
\Omega^1_Y(\log D)$.

To give an idea of the proof, consider the simple setup where $Y = Y^\circ$ is
compact and admits a morphism $\gamma: Y \to Z$ to a curve such that the
family $f^\circ$ is the pull-back of a smooth family that lives over
$Z$. Applied to the family over the one-dimensional space $Z$, the Viehweg-Zuo
result implies that $\Omega^1_Z$ is ample, so that the inclusion
$\gamma^*\Omega^1_Z \subseteq \Omega^1_Y$ immediately shows that $Y$ cannot be
special. Since all sheaves constructed by Viehweg and Zuo really come from the 
moduli space, a more elaborate version of this argument can in fact be used to
deal with all cases of Theorem~\ref{thm:main} where the moduli map has a
one-dimensional image. For moduli maps with higher-dimensional images, minimal
model theory gives the link between the existence of $\sA$ and positivity of
subsheaves in $\Omega^1_Y(\log D)$.

\subsection*{Acknowledgments}

Conjecture~\ref{conj:campana} was brought to our attention by Fr\'ed\'eric
Campana during the 2007 Levico conference in Algebraic Geometry. We would like
to thank him for a number of discussions on the subject.

\part{$\mathbb{\cC}$-PAIRS AND THEIR DIFFERENTIALS}\label{P1}

\section{$\cC$-pairs, adapted morphisms and covers}
\label{sec:orbifold}

\subsection{$\cC$-pairs, introduction and definitions}
\label{sec:cPintro}

Let $\gamma: Y \to X$ be a finite morphism of degree $N$ between
$n$-dimensional smooth varieties and assume that $\gamma$ is totally branched
over a smooth divisor $D \subset X$. In this setting, if $\sigma \in
\Gamma\bigl( X,\, \Omega^p_X(*D) \bigr)$ is a $p$-form, possibly with poles of
arbitrary order\footnote{See Definition~\ref{defn:arbitraryPoles} on
  page~\pageref{defn:arbitraryPoles} for a proper definition of the sheaf
  $\Omega^p_X(*D)$ of differential forms with poles of arbitrary order along
  $D$.} along $D$, its pull-back $\gamma^*(\sigma)$ is again a $p$-form, now
with poles along $D_\gamma := \supp \gamma^*(D)$. It is an elementary fact
that to check whether $\sigma$ does indeed have poles, it suffices to look at
its pull-back $\gamma^*(\sigma)$. More precisely, it is true that $\sigma$ has
poles of positive order if and only if $\gamma^*(\sigma)$ does. A similar
statement holds for forms with logarithmic poles along $D$. This is, however,
no longer true if we look at symmetric products of $\Omega^p_X$.

For an example that will be important later, choose local coordinates $z_1,
\ldots, z_n$ on $X$ such that $D = \{z_1 =0\}$. The symmetric form
\begin{equation}\label{eq:sfs}
  \sigma := \frac{1}{z^a_1} (dz_1)^{\otimes b_1} \otimes (dz_2)^{\otimes
    b_2} \otimes \cdots \otimes (dz_n)^{\otimes b_n} \in \Gamma\bigl(
  X,\, \Sym^{b_1+\cdots+b_n}\Omega^1_X(*D) \bigr)
\end{equation}
has a pole of order $a$ along $D$. However, an elementary computation shows
that $\gamma^*(\sigma)$ does not have any pole if the pole order of $\sigma$
is sufficiently small with respect to $b_1$, that is $a \leq b_1 \cdot
\frac{N-1}{N}$.

In our proof of Theorem~\ref{thm:main}, we consider morphisms $\gamma: Y \to
X$, where $X$ is a suitable subvariety of the coarse moduli space and $Y$ is a
chart for the moduli stack, or simply has a morphism to the moduli
stack. Tensor products of $\Omega^p_Y$ and $\Omega^p_X$ and the pull-back map
appear naturally in this context when one discusses positivity and the
Kodaira-Iitaka dimension of invertible subsheaves of $\Omega^p_Y$, and tries
to relate that to objects living on the coarse moduli space.  The formal
set-up for this discussion has been given by Campana in his theory of
\emph{Orbifoldes G\'eom\'etriques}.  Since the word \emph{orbifold} is already
used in a different context, and since the notion of a \emph{geometric
  orbifold} is not widely accepted, we have chosen to use the words
\emph{$\cC$-pair}, \emph{$\cC$-form} and \emph{$\cC$-differential} in this
paper, where ``$\cC$'' stands for Campana. In this language, we will say that
the form $\sigma$ defined in~\eqref{eq:sfs} is a $\cC$-form on the $\cC$-pair
$(X, \frac{N-1}{N} \cdot D)$ if and only if $a \leq b_1 \cdot \frac{N-1}{N}$
holds.

\begin{notation}
  We will often need to consider numbers $\frac{N-1}{N}$, where $N$ is either
  a positive integer or $N = \infty$. Throughout the paper we follow the
  convention that $\frac{\infty-1}{\infty} := 1$.
\end{notation}

\begin{defn}[\protect{$\cC$-pair and $\cC$-multiplicities, cf.~\cite[Def.~2.1]{Cam07}}]\label{def:orbifold}
  A \emph{$\cC$-pair} is a pair $(X, D)$ where $X$ is a normal variety or
  complex space and $D$ is a $\Q$-divisor of the form
  $$
  D = \sum_i {\textstyle \frac{n_i-1}{n_i}} \cdot D_i
  $$
  where the $D_i$ are irreducible and reduced distinct Weil divisors on $X$
  and $n_i \in \bN^+ \cup \{\infty\}$. The numbers $n_i$ are called
  \emph{$\cC$-multiplicities} of the components $D_i$, denoted
  $m_{(X,D)}(D_i)$. More generally, if $E \subset X$ is any irreducible,
  reduced Weil divisor, set
  $$
  m_{(X,D)}(E) := \left\{ 
    \begin{matrix}
      n_i & \text{if $\exists i$ such that $E = D_i$ } \\
      1 & \text{otherwise}
    \end{matrix}
  \right.
  $$
\end{defn}

\subsection{Adapted coordinates}

In Section~\ref{sec:differentials}, we compute sheaves of $\cC$-differentials
in local coordinates.  For this, we consider ``adapted'' systems of
coordinates, defined as follows.

\begin{defn}[Adapted coordinates]\label{def:adaptedCoordinates}
  Let $(X, D)$ be a $\cC$-pair, and let $x \in \supp(D)$ be a point which is
  smooth both in $X$ and in $\supp(D)$. If $U$ is a neighborhood of $x$, open
  in the analytic topology, and if $z_1, \ldots, z_n \in \sO_{\hol}(U)$ are
  local analytic coordinates about $x$, we say that the $z_i$ form an
  \emph{adapted system of coordinates} if the set-theoretic equation
  $$
  \supp(D) \cap U = \{z_1 = 0\}
  $$
  holds.
\end{defn}

\begin{rem}\label{rem:wherearecoords}
  If $(X, D)$ is a $\cC$-pair, and $x \in \supp(D)$ is a point which is smooth
  both in $X$ and in $\supp(D)$, then there always exists an open neighborhood
  of $x$ with an adapted system of coordinates. The set of points for which
  there is no system of adapted coordinates is therefore contained in a closed
  subset of codimension $\geq 2$.
\end{rem}

The last remark shows that the set of points for which there is no system of
adapted coordinates will not play any role when we use adapted coordinates in
the discussion of \emph{reflexive} sheaves of differentials. For a more
general setup on smooth spaces, see \cite[Sect.~2.5]{Cam07}.

\subsection{Adapted morphisms}

In Section~\ref{sec:cPintro}, we attached a $\mathcal C$-pair to the base of a
finite morphism. Conversely, in the discussion of a given $\cC$-pair $(X, D)$,
we will often use morphisms $Y \to X$ which induce the $\mathcal C$-pair
structure on $X$, at least to some extent. In this section, we introduce the
necessary notation and prove the existence of these ``adapted'' morphisms.

\begin{notation}[Multiplicity of a Weil divisor in a pull-back divisor]\label{not:pbm}
  Let $\gamma: Y \to X$ be a surjective morphism of normal varieties of
  constant fiber dimension. If $D$ is any divisor on $X$, its restriction
  $D|_{X_{\reg}}$ to the smooth locus of $X$ is Cartier.  In particular, there
  exists a pull-back $\gamma^*(D|_{X_{\reg}})$, which we can interpret as a
  Weil divisor on the normal space $\gamma^{-1}(X_{\reg})$. If $E \subset Y$
  is any irreducible divisor, then $E$ necessarily intersects
  $\gamma^{-1}(X_{\reg})$, and it makes sense to consider the coefficient $m$
  of the pull-back divisor $\gamma^*(D|_{X_{\reg}})$ along
  $E|_{\gamma^{-1}(X_{\reg})}$. Abusing notation, we say that \emph{$E$
    appears in $\gamma^*(D)$ with multiplicity $m$}.
\end{notation}

\begin{convention}[Pull-back of Weil divisors]
  In the setup of Notation~\ref{not:pbm}, the pull-back morphism for Cartier
  divisors defined on $X_{\reg}$ extends to a well-defined pull-back morphism
  $$
  \gamma^* : \{ \text{Weil divisors on }X\} \to \{ \text{Weil divisors
    on }Y\}
  $$
  that respects linear equivalence. Throughout this article, whenever a
  surjective morphism of constant fiber dimension is given, we will use the
  pull-back morphism for Weil divisors and their linear equivalence classes
  without extra mention.
\end{convention}

\begin{defn}[Adapted morphism]\label{def:adaptedmorphism}
  Let $(X, D)$ be a $\cC$-pair, with $D = \sum_i \frac{n_i-1}{n_i}D_i$. A
  surjective morphism $\gamma: Y \to X$ from an irreducible and normal space
  is called \emph{adapted} if the following holds:
  \begin{enumerate}
  \item\ilabel{il:adpt} for any number $i$ with $n_i < \infty$ and any
    irreducible divisor $E \subset Y$ that surjects onto $D_i$, the divisor
    $E$ appears in $\gamma^*(D_i)$ with multiplicity precisely $n_i$.
  \item the fiber dimension is constant on $X$.
  \end{enumerate}
  The morphism $\gamma$ is called \emph{subadapted} if in \iref{il:adpt} we
  require only that $E$ appears in $\gamma^*(D_i)$ with multiplicity at least
  $n_i$.
\end{defn}

The preimage of the logarithmic part of $D$ will appear again and again when
we use adapted covers to discuss the differentials associated with a $\mathcal
C$-pair. We will thus introduce a specific notation for this divisor.

\begin{notation}[Adapted logarithmic divisor]\label{not:dgamma}
  Given a $\cC$-pair $(X, D)$ and an adapted or subadapted morphism $\gamma: Y
  \to X$ as in Definition~\ref{def:adaptedmorphism}, we set
  $$
  D_{\gamma} := \supp \gamma^*(\lfloor D \rfloor) \subset Y.
  $$
  We call $D_\gamma$ the \emph{adapted logarithmic divisor} associated with
  $\gamma$.
\end{notation}

Given a $\cC$-pair $(X,D)$ as in Definition~\ref{def:orbifold} and general
hyperplane $H$, we construct an adapted morphism $\ga: Y \to X$ which is also
a finite cyclic cover totally branched over $H$. The proof is fairly standard
and is included only for completeness.

\begin{prop}[Existence of an adapted morphism]\label{prop:existencesofcovers}
  Let $(X, D)$ be a $\cC$-pair as in Definition~\ref{def:orbifold}. If $X$ is
  projective and if the components $D_i \subseteq D$ are Cartier, then there
  exists a very ample line bundle $L \in \Pic(X)$ such that for general $H \in
  |L|$, there exists a finite cover $\gamma: Y \to X$ with the following
  properties.
  \begin{enumerate}
  \item The domain $Y$ is normal.
  \item The morphism $\gamma$ is adapted in the sense of
    Definition~\ref{def:adaptedmorphism}. It is cyclic, in particular Galois.
  \item The branch locus of $\gamma$ is the union of $H$ and those components
    $D_i$ with $\cC$-multiplicities $n_i \ne \infty$.
  \item The morphism $\gamma$ is totally branched over $H$.
  \item If $D_i \subseteq \lfloor D \rfloor$ is any component, then $\gamma$
    in unbranched over the general point of $D_i$.
  \end{enumerate}
\end{prop}
\begin{proof}
  For convenience of notation, we sort the indices $n_i$ so that the first
  $\cC$-multiplicities $n_1, n_2, \ldots, n_k$ are those that are finite. Let
  $N$ be the least common multiple of the $\cC$-multiplicities $(n_i)_{i\leq
    k}$ that are not $\infty$, consider a very ample Cartier divisor $A$ such
  that
  $$
  L := A^{\otimes N}- \sum_{i \leq k} {\textstyle \frac{N}{n_i}} \cdot D_i
  $$
  is still very ample, and consider a general hyperplane $H \in |L|$. Let
  $\sigma \in H^0\bigl(X, A^{\otimes N} \bigr) \setminus \{ 0 \}$ be a
  non-vanishing section associated to the divisor $H + \sum_{i \leq k}
  \frac{N}{n_i} \cdot D_i \in |A^{\otimes N}|$. Abusing notation, let $A$ and
  $A^{\otimes N}$ also denote the total spaces of the associated bundles.
  Consider the multiplication map $m : A \to A^{\otimes N}$, identify the
  section $\sigma$ with a subvariety of the space $A^{\otimes N}$, and let
  $\tilde \sigma \subset A$ be the preimage $\tilde \sigma =
  m^{-1}(\sigma)$. The map $m|_{\tilde \sigma} : \tilde \sigma \to \sigma$ is
  clearly a cyclic cover, with an associated action of $\factor \mathbb
  Z.N\mathbb Z.$, acting via multiplication with $N^{\rm{th}}$ roots of unity.

  The restricted morphism $m|_{\tilde \sigma} : \tilde \sigma \to \sigma$ is
  obviously unbranched away from $H \cup \bigcup_{i \leq k} D_i$. Over the
  general point of $H$, the variety $\tilde \sigma$ is smooth and the morphism
  $m|_{\tilde \sigma}$ is totally branched to order $N$.
  
  Now let $x$ be a general point of one of the $D_i$ with $i \leq k$. Choose
  an open neighborhood of $x$ with a system of adapted coordinates, $z_1,
  \ldots, z_n$, and choose bundle coordinates $y$ and $y'$ on $A$ and
  $A^{\otimes N}$, respectively, such that the multiplication map $m$ is given
  as $y \mapsto y^N = y'$. In these coordinates, we have $D_i = \{z_1=0 \}$,
  and the subvarieties $\sigma$ and $\tilde \sigma$ are given as
  $$
  \sigma = \left\{ y' - z_1^{\frac{N}{n_i}} = 0 \right\} \quad \text{and}
  \quad \tilde \sigma = \left\{ y^N - z_1^{\frac{N}{n_i}} = 0 \right\}.
  $$
  Recalling that
  $$
  y^N - z_1^{\frac{N}{n_i}} = \left( y^{n_i} \right)^{\frac{N}{n_i}} -
  z_1^{\frac{N}{n_i}} = \prod_{k=0}^{\frac{N}{n_i}-1} \left( y^{n_i} -
    \varepsilon^k \cdot z_1 \right)
  $$
  for $\varepsilon = \exp \bigl( \frac{n_i}{N}\cdot \sqrt{-1} \bigr)$, similar
  to \cite[Sect.~III.9]{HBPV}, we obtain that
  $$
  \tilde \sigma = \bigcup_{k=0}^{\frac{N}{n_i}-1} \left\{ y^{n_i} =
    \varepsilon^k \cdot z_1 \right\}
  $$
  is the union of $\frac{N}{n_i}$ distinct smooth components, each totally
  branched to order $n_i$ over $D_i$. Defining $Y$ as the normalization of
  $\tilde \sigma$, we obtain the claim.
\end{proof}

\begin{notation}[Cyclic adapted cover with extra branching]\label{not:adptweb}
  Given a $\cC$-pair $(X, D)$ and a general hyperplane $H$ as in
  Proposition~\ref{prop:existencesofcovers}, we call the associated morphism
  $\gamma$ a \emph{cyclic adapted cover with extra branching along $H$} and
  set $H_{\gamma} := \supp \gamma^*(H)$.
\end{notation}

The standard adjunction formula immediately gives the following useful
relation between the log canonical divisor $K_Y + D_{\gamma}$ and the
pull-back of $K_X + D$.

\begin{lem}\label{lem:adjunction}
  If $\gamma : Y \to X$ is a cyclic adapted cover with extra branching along
  $H$, the following equivalence of Weil divisor classes holds,
  $$
  K_Y + D_{\gamma} = \gamma^*(K_X + D) + (N-1)\cdot H_{\gamma},
  $$
  where $N$ is the degree of the finite morphism $\gamma$.
\end{lem}
\begin{proof}
  Again, we sort the indices $n_i$ so that the first $\cC$-multiplicities
  $n_1, n_2, \ldots, n_k$ are those that are finite. By definition of adapted
  cover, the cycle-theoretic preimage $\gamma^*(D_i)$ is a sum of divisors
  $D_{i,j}$ that appear with multiplicity precisely $n_i$ if $i \leq k$, and
  with multiplicity one if $i > k$
  $$
  \gamma^*(D_i) = \left\{ 
    \begin{matrix}
      \sum_j n_i\cdot D_{i,j} & \text{if $i \leq k$} \\
      \sum_j  D_{i,j} & \text{if $i > k$}
    \end{matrix} \right.
  $$
  In particular,
  $$
  \gamma^*(D) = \sum_{i\leq k} \sum_j {\textstyle n_i\frac{n_i-1}{n_i}} \cdot D_{i,j} +
  \sum_{i>k}\sum_j D_{i,j} = \sum_{i\leq k} \sum_j (n_i-1)\cdot D_{i,j} +
  D_{\gamma}.
  $$
  Together with the standard adjunction formula for a finite morphism,
  $$
  K_Y = \gamma^*(K_X) + \sum_{i \leq k}\sum_j (n_i-1)\cdot D_{i,j} +
  (N-1)\cdot H_{\gamma},
  $$
  this gives the claim.
\end{proof}

\subsection{Adapted differentials}

If $\gamma : Y \to X$ is a cyclic adapted cover with extra branching along $H$
and if $X$ and $Y$ are smooth, it will be useful later to slightly enlarge the
sheaf $\gamma^* \Omega^1_X(\log \lfloor D \rfloor)$ and consider a sheaf
$\Omega^1_Y(\log D_{\gamma})_{\rm adpt}$ of ``adapted differentials'' with
$$
\det \Omega^1_Y(\log D_{\gamma})_{\rm adpt} \cong \sO_Y \bigl(\gamma^*(K_X +
D) \bigr).
$$
If $X$ and $Y$ are singular, we do a similar construction, using the reflexive
hull of $\gamma^* \Omega^1_X(\log \lfloor D \rfloor)$. The following notation
is useful in this context and is used throughout the present paper.

\begin{notation}[Reflexive sheaves and operations]
  Let $Z$ be a normal variety and $\sA$ a coherent sheaf of
  $\O_Z$-modules. For $n\in \bN$, set $\sA^{[n]} := \otimes^{[n]} \sA :=
  (\sA^{\otimes n})^{**}$, $\Sym^{[n]} \sA := (\Sym^n \sA)^{**}$,
  etc. Likewise, for a morphism $\gamma : X \to Z$ of normal varieties, set
  $\gamma^{[*]}\sA := (\gamma^*\sA)^{**}$. If $\sA$ is reflexive of rank one,
  we say that $\sA$ is $\Q$-Cartier if there exists a number $n\in\bN$ such
  that $\sA^{[n]}$ is invertible.
\end{notation}

Adapted differentials are now defined as follows.

\begin{defn}[Adapted differentials]\label{def:adptdiffs}
  If $\gamma : Y \to X$ is a cyclic adapted cover with extra branching along
  $H$ and $1 \leq p \leq \dim X$, we define a sheaves
  $$
  \Omega^{[p]}_{Y}(\log D_{\gamma})_{\rm adpt} \subseteq \Omega^{[p]}_{Y}(\log
  D_{\gamma}),
  $$
  called \emph{sheaves of adapted differentials associated with the adapted
    cover $\gamma$}, on the level of presheaves as follows. If $U \subseteq Y$
  is any open set and $\sigma \in \Gamma \bigl( U,\, \Omega^{[p]}_{Y}(\log
  D_{\gamma}) \bigr)$ any section, then $\sigma$ is in $\Gamma \bigl( U,\,
  \Omega^{[p]}_{Y}(\log D_{\gamma})_{\rm adpt} \bigr)$ if and only if the
  restriction of $\sigma$ to the open set $V := U \setminus \gamma^{-1}(\lceil
  D \rceil)$ satisfies $\sigma|_V \in \GA \bigl( V, \,
  \gamma^{[*]}\Omega^{[p]}_X \bigr)$.
\end{defn}

We end this section by noting a few properties of the sheaf of adapted
differentials for later use.

\begin{rem}[Reflexivity, inclusions of adapted differentials]
  It is immediate from the definition that the sheaf $\Omega^{[p]}_{Y}(\log
  D_{\gamma})_{\rm adpt}$ of adapted differentials is reflexive. Since
  $\gamma^{[*]} \bigl( \Omega^{[p]}_X(\log \lfloor D \rfloor) \bigr) \subseteq
  \Omega^{[p]}_Y(\log D_{\gamma})$, it is also clear that there exist
  inclusions
  $$
  \gamma^{[*]} \bigl( \Omega^{[p]}_{X}(\log \lfloor D \rfloor) \bigr)
  \subseteq \Omega^{[p]}_{Y}(\log D_{\gamma})_{\rm adpt} \subseteq
  \Omega^{[p]}_{Y}(\log D_{\gamma}).
  $$
\end{rem}

\begin{rem}[Determinant of adapted differentials]\label{rem:c1ofadptdiffs}
  There exist isomorphisms of sheaves
  \begin{align*}
    \det \left( \Omega^{[1]}_{Y}(\log D_{\gamma})_{\rm adpt} \right) & \cong \sO_Y
    \bigl(K_Y + D_{\gamma} - (N-1)\cdot H_{\gamma}\bigr) && \text{by Construction}\\
    & \cong \sO_Y \bigl(\gamma^*(K_X + D) \bigr). && \text{by Lemma~\ref{lem:adjunction}}
  \end{align*}
\end{rem}

\begin{rem}[Normal bundle sequence for adapted differentials]\label{rem:molitor}
  Let $F \subset X$ be a smooth curve. Assume that the pair $(X, \lceil D
  \rceil \cup H)$ is snc along $F$, and that $F$ intersects the support
  $\supp(D+ H)$ transversely. The preimage $\tilde F:= \gamma^{-1} (F) \subset
  Y$ is then smooth, intersects $D_\gamma \cup H_\gamma$ transversely, and the
  standard conormal sequence of logarithmic differentials,
  $$
  0 \to N^*_{\tilde F/Y} \to \Omega^1_Y(\log D_\gamma)|_{\tilde F} \to
  \Omega^1_{\tilde F}(\log D_\gamma|_{\tilde F}) \to 0,
  $$
  restricts to an exact sequence
  $$
  0 \to \underbrace{N^*_{\tilde F/Y}}_{\makebox[0mm][l]{\scriptsize $\cong
      \gamma^*(N^*_{F/X})$}} \to \Omega^1_Y(\log D_\gamma)_{\rm adpt}|_{\tilde
    F} \to \underbrace{\Omega^1_{\tilde F}(\log D_\gamma|_{\tilde F})\otimes
    \sO_{\tilde F}\bigl(-(N-1)H_\gamma|_{\tilde F}\bigr)}_{\cong
    \gamma^*(\Omega^1_F)\otimes \sO_{\tilde F}(\gamma^*D|_F)} \to 0.
  $$
\end{rem}

\section{$\cC$-differentials}
\label{sec:differentials}

Given a $\cC$-pair $(X,D)$ and numbers $p$ and $d$, we next define the sheaf
of $\cC$-differentials, written as $\Sym_{\cC}^{[d]} \Omega^p_X(\log D)$. A
section $\sigma \in \Gamma \bigl( X,\, \Sym_{\cC}^{[d]} \Omega^p_X(\log D)
\bigr)$ is a symmetric form on $X$, possibly with logarithmic poles along the
support of $D$, which satisfies extra conditions.  There are two essentially
equivalent ways to specify what these conditions are.

\begin{enumerate}
\item\ilabel{il:Campana} The pole order of $\sigma$ along a component of $D$
  is small compared to the multiplicity of the component in $D$, and to the
  pole order of forms $f \cdot \sigma \in \Gamma \bigl( X,\, \Sym_{\cC}^{[d]}
  \Omega^p_X(\log D) \bigr)$, where $f$ is a rational or meromorphic function.
\item\ilabel{il:us} The pull-back of $\sigma$ to any adapted covering $\gamma$
  has only logarithmic poles along $D_\gamma$, and no other poles elsewhere.
\end{enumerate}

The sheaf of $\cC$-differentials has been defined in \cite{Cam07} writing down
Condition~\iref{il:Campana} in adapted coordinates on smooth spaces. For our
purposes, however, Condition~\iref{il:us} is more convenient. The relation
between the definitions is perhaps most clearly seen when the
$\cC$-differentials are computed explicitly in local coordinates. This is done
in Computation~\ref{comp:obri} below.

\subsection{Useful results of sheaf theory}

Before defining the sheaf of $\cC$-differentials in
Definition~\ref{def:orbifoldDiff} below, we recall a few facts and definitions
concerning saturated and reflexive sheaves.

\begin{defn}[Saturation of a subsheaf]\label{def:saturation}
  Let $X$ be a normal variety, $\sB$ a coherent, reflexive sheaf of
  $\sO_X$-modules and $\sA$ a subsheaf, with inclusion $\iota: \sA \to \sB$.
  The saturation of $\sA$ in $\sB$ is the kernel of the natural map
  $$
  \sB \to \factor \coker(\iota).\tor..
  $$
  If the ambient sheaf $\sB$ is understood from the context, the saturation of
  $\sA$ is often denoted as $\overline{\sA}$. If $\coker(\iota)$ is torsion
  free, we say that $\sA$ is saturated in $\sB$.
\end{defn}

\begin{prop}[\protect{Reflexivity of the saturation, cf.~\cite[Claim on p.~158]{OSS}}]\label{prop:satisrefl}
  In the setup of Definition~\ref{def:saturation}, the saturation
  $\overline{\sA}$ is reflexive. \qed
\end{prop}

The next proposition shows that the reflexive symmetric product of a saturated
sheaf remains saturated.

\begin{prop}[Saturation and symmetric products]\label{prop:satandsym}
  Let $X$ be a normal variety, $\sB$ a coherent, reflexive sheaf of
  $\sO_X$-modules and $\sA$ a saturated subsheaf, with inclusion $\iota: \sA
  \to \sB$. If $m$ is any number, then the natural inclusion of reflexive
  symmetric products,
  $$
  \Sym^{[m]}\iota : \Sym^{[m]} \sA \to \Sym^{[m]} \sB
  $$
  represents $\Sym^{[m]} \sA$ as a saturated subsheaf of $\Sym^{[m]} \sB$.
\end{prop}
\begin{proof}
  There exists a closed subset $Z \subset X$ of $\codim_X Z \geq 2$ such that
  $\sA$, $\sB$ and $\coker(\iota)$ are locally free on $X^\circ := X \setminus
  Z$. It follows from standard sequences \cite[II,~Ex.~5.16]{Ha77} that the
  cokernel of $\Sym^{[m]}\iota$ is torsion-free on $X^\circ$. In particular,
  the natural inclusion
  \begin{equation}\label{eq:mbr}
    \Sym^{[m]}\sA \to \overline{\Sym^{[m]}\sA}    
  \end{equation}
  is isomorphic away from $Z$. By definition and by
  Proposition~\ref{prop:satisrefl}, respectively, both sides of~\eqref{eq:mbr}
  are reflexive. The inclusion~\eqref{eq:mbr} must thus be isomorphic.
\end{proof}

\begin{defn}[Sheaf of sections with arbitrary pole order]\label{defn:arbitraryPoles}
  Let $X$ be a variety, let $D \subset X$ be a reduced Weil divisor and $\sF$
  a reflexive coherent sheaf of $\sO_X$-modules. We will often consider
  sections of $\sF$ with poles of arbitrary order along $D$, and let $\sF(*D)$
  be the associated sheaf of these sections.  More precisely, we define
  $$
  \sF(*D) := \lim_{\overset{\longrightarrow}{m}} \bigl( \sF \otimes
  \sO_X(m\cdot D) \big)^{**}.
  $$
\end{defn}

\subsection{The definition of $\pmb \cC$-differentials}

We next define a $\cC$-differential.  Our approach is slightly different than
Campana's approach in \cite{Cam07}, as Campana defines $\cC$-differentials in
local adapted coordinates.  However, we will recover his definition in
Section~\ref{subsec:localcoord}.

\begin{defn}[\protect{$\cC$-differentials, cf.~\cite[Sect.~2.6-7]{Cam07}}]\label{def:orbifoldDiff}
  If $(X,D)$ is a $\cC$-pair we define a sheaf
  $$
  \underbrace{ \Sym_{\cC}^{[d]} \Omega^p_X(\log D)}_{=:\sA } \subseteq
  \underbrace{\bigl(\Sym^{[d]} \Omega^p_X\bigr) (* \lceil D \rceil)}_{=: \sB}
  $$
  on the level of presheaves as follows: if $U \subseteq X$ is open and
  $\sigma \in \Gamma \bigl( U,\, \sB \bigr)$ any form, possibly with poles
  along $D$, then $\sigma$ is a section of $\sA$ if and only if for any open
  subset $U' \subseteq U$ and any adapted morphism $\gamma: V \to U'$, the
  reflexive pull-back has at most logarithmic poles along $D_\gamma$, and no
  other poles elsewhere, i.e.
  \begin{equation}\label{eq:orbifoldDiff}
    \gamma^{[*]}(\sigma) \in \Gamma \bigl( V,\, \Sym^{[d]} 
    \Omega^p_V(\log D_{\gamma} )\bigr).
  \end{equation}
\end{defn}

\begin{explanation}
  Inclusion~\eqref{eq:orbifoldDiff} of Definition~\ref{def:orbifoldDiff} can
  also be expressed as follows. If $E \subset V$ is any irreducible Weil
  divisor which dominates a component of $\lfloor D \rfloor$, then
  $\gamma^*(\sigma)$ may have at most logarithmic poles along $E$.  If $E$
  does not dominate a component of $\lfloor D \rfloor$, then
  $\gamma^*(\sigma)$ may not have any poles along $E$.
\end{explanation}

\begin{rem}
  Definition~\ref{def:orbifoldDiff} remains invariant if we remove arbitrary
  small sets from $U'$.  It is therefore immediate that the sheaf
  $\Sym_{\cC}^{[d]} \Omega^p_X(\log D)$ is torsion free and normal as a sheaf
  of $\sO_X$-modules, cf.~\cite[Def.~1.1.11 on p. 150]{OSS}. Once we have seen in
  Corollary~\ref{cor:reflexivity} that $\Sym_{\cC}^{[d]} \Omega^p_X(\log D)$
  is also coherent, this will imply that it is in fact reflexive.
\end{rem}

\subsection{$\pmb \cC$-differentials in local coordinates}
\label{subsec:localcoord}

It is sometimes useful to represent $\cC$-differentials explicitly in local
coordinates. The following computations yields several results which will be
needed later on.

\begin{computation}\label{comp:obri}
  Let $(X, D)$ be a $\cC$-pair as in Definition~\ref{def:orbifold}.  Let $D_i
  \subseteq D$ be a component, let $x \in D_i$ be a smooth point, and let $U \subseteq X$ be an open
  neighborhood of $x$ with an adapted system of coordinates as in
  Definition~\ref{def:adaptedCoordinates}. Finally, consider a section
  $$
  \sigma := \frac{f(z_1, \ldots, z_n)}{z_1^a}\cdot (dz_1)^{m_1}\cdot
  (dz_2)^{m_2} \cdots (dz_n)^{m_n} \in \Gamma \left( U,\, \Sym^{[d]}
    \Omega^1_X \bigl(* \lceil D \rceil \bigr)\right),
  $$
  where $d = \sum m_i$ and $f \in \sO_U$ is a holomorphic function that does
  not vanish along $D_i \cap U = \{ z_1 = 0\}$. We aim to express
  Condition~\eqref{eq:orbifoldDiff} in this context. To this end, after
  possibly replacing $U$ by one of its open subsets, let $\gamma: V \to U$ be
  any adapted morphism, and $E \subset V$ any divisor that dominates $D_i \cap
  U$.

  If $D_i$ appears in $D$ with $\cC$-multiplicity $n_i = \infty$, it is a
  standard fact that $\gamma^{[*]}(\sigma)$ has logarithmic poles along $E$ if
  and only if $\sigma$ has logarithmic poles along $D_i$, see
  e.g.~\cite[Cor.~2.12.1]{GKK08}. Condition~\eqref{eq:orbifoldDiff} therefore
  says that $\sigma$ is a section of $\Sym^{[d]}_{\cC} \Omega^1_X(\log D)$ if
  and only if $a \leq m_1$.

  If $D_i$ appears in $D$ with $\cC$-multiplicity $n_i < \infty$, then $E$
  appears in $\gamma^*(D_i)$ with multiplicity $n_i$. The reflexive pull-back
  $\gamma^{[*]}(\sigma)$ is thus a rational section of the sheaf $\Sym^{[d]}
  \Omega^1_V(\log D_{\gamma})$ whose pole order along $E$ is precisely
  \begin{equation}\label{eq:POs}
    P(\sigma, D_i) := n_i \cdot a - (n_i-1)\cdot m_1.
  \end{equation}
  We obtain from Condition~\eqref{eq:orbifoldDiff} that $\sigma$ is a section
  of $\Sym^{[d]}_{\cC} \Omega^1_X(\log D)$ if and only if $P(\sigma, D_i) \leq
  0$.
\end{computation}

\begin{computation}\label{comp:obri2}
  In the setup of Computation~\ref{comp:obri}, if $\tau$ is an arbitrary
  section of $\bigl(\Sym^{[d]} \Omega^1_X\bigr) (* \lceil D \rceil)$, write
  $\tau$ locally as
  $$
  \tau := \sum_{m_1\ldots m_n} \underbrace{\frac{f_{m_1\ldots m_n}(z_1,
      \ldots, z_n)}{z_1^{a_{m_1\ldots m_n}}}\cdot (dz_1)^{m_1} \cdot
    (dz_2)^{m_2} \cdots (dz_n)^{m_n}}_{=:\, \sigma_{m_1\ldots m_n}},
  $$
  where the functions $f_{m_1\ldots m_n}$ are either constantly zero, or do
  not vanish along $D_i \cap U$. Again, we aim to formulate
  Condition~\eqref{eq:orbifoldDiff} for the section $\tau$. Choose an adapted
  covering $\gamma$ and a divisor $E$ as in Computation~\ref{comp:obri}.

  If $D_i$ appears in $D$ with $\cC$-multiplicity $n_i = \infty$, it is again
  clear that $\gamma^{[*]}(\tau)$ has logarithmic poles along $E$ if and only if
  $\tau$ has logarithmic poles along $D_i$. Condition~\eqref{eq:orbifoldDiff}
  therefore says that $\tau$ is a section of $\Sym^{[d]}_{\cC} \Omega^1_X(\log
  D)$ if and only if $a_{m_1\ldots m_n} \leq m_1$ for all multi-indices $m_1,
  \ldots, m_n$ with $f_{m_1\ldots m_n} \not \equiv 0$.

  If $D_i$ appears in $D$ with $\cC$-multiplicity $n_i < \infty$, set
  $$
  P(\tau, D_i) = \max \bigl\{ P(\sigma_{m_1\ldots m_n}, D_i) \,|\,
  f_{m_1\ldots m_n} \not \equiv 0 \bigr\},
  $$
  where the $P(\sigma_{m_1\ldots m_n}, D_i)$ are the numbers defined in
  Equation~\eqref{eq:POs} above.  It is then clear that the reflexive
  pull-back $\gamma^{[*]}(\tau)$ is a rational section of the sheaf
  $\Sym^{[d]} \Omega^1_V(\log D_{\gamma})$ whose pole order along $E$ is
  precisely $P(\tau, D_i)$. Again, we obtain from
  Condition~\eqref{eq:orbifoldDiff} that $\tau$ is a section of
  $\Sym^{[d]}_{\cC} \Omega^1_X(\log D)$ if and only if $P(\tau, D_i) \leq 0$.
\end{computation}

\begin{obs}\label{obs:lfetc}
  Using the convention that $\frac{n_i -1}{n_i} = 1$ if $n_i = \infty$,
  Computation~\ref{comp:obri2} gives the following set of generators for
  $\Sym^{[d]}_{\cC}\Omega^1_U(\log D)$ near the point $x$,
  $$
  \left( \frac{1}{z_1^{\lfloor m_1 \cdot \frac{n_i-1}{n_i} \rfloor}}\cdot
    (dz_1)^{m_1} \cdot (dz_2)^{m_2} \cdots (dz_n)^{m_n} \right)_{\sum m_j = d.}
  $$
  Thus, it follows from Definition~\ref{def:adaptedCoordinates} that the sheaf
  $\Sym^{[d]}_{\cC}\Omega^1_X(\log D)$ is locally free wherever
  the pair $(X, \lceil D \rceil)$ is snc. In particular, it is locally free in
  codimension one. Since it is normal, we also see that
  \begin{equation}\label{eq:scincl}
    \Sym^{[d]}_{\cC} \Omega^1_X(\log D) \subseteq \Sym^{[d]} \Omega^1_X(\log
    \lceil D \rceil).
  \end{equation}
  In the case $d = 1$, we obtain additionally that $\Sym^{[1]}_{\cC}
  \bigl(\Omega^1_X(\log D)\bigr) = \Omega^{[1]}_X (\log \lfloor D\rfloor)$.
\end{obs}

\begin{obs}\label{obs:Pindep}
  In Computation~\ref{comp:obri2}, if $n_i < \infty$, the number $P(\tau,
  D_i)$ depends only on the section $\tau$ and on the component $D_i$, but not
  on the choice of adapted coordinates, or on the choice of the adapted
  morphism $\gamma$.
\end{obs}

\begin{obs}\label{obs:Pneg}
  In Computation~\ref{comp:obri2}, if $n_i < \infty$ and if the number
  $P(\tau, D_i)$ is non-positive, then $\gamma^{[*]}(\tau)$ is a section of
  the sheaf $\Sym^{[d]} \Omega^1_V(\log D_{\gamma})$ that vanishes along $E$
  precisely with multiplicity $-P(\tau, D_i)$.
\end{obs}

\subsection{Consequences of the local computation}

Computations~\ref{comp:obri} and \ref{comp:obri2} have several immediate
consequences which we note for future reference. It is not very hard to see
that the computations and observations of Section~\ref{subsec:localcoord} also
hold for sections in $\Sym^{[d]}_{\cC} \Omega^p_X(\log D)$, for all numbers
$p$. The consequences of Computation~\ref{comp:obri} which we draw in this
section also hold for all $p$, and are stated in that generality. To keep the
paper reasonably sized, we leave it to the reader to make the analogous
computations in case $p \not = 1$.

\subsubsection{Inclusions, reflexivity}

In complete analogy to Inclusion~\eqref{eq:scincl} above, we can view the
sheaf of $\cC$-differentials as a subsheaf of the logarithmic differentals,
for any $p$. In Corollary~\ref{cor:reflexivity}, we apply this inclusion to
prove reflexivity of the sheaf of $\cC$-differentials.

\begin{cor}[Inclusion of $\cC$-differentials into logarithmic differentials]\label{cor:inclusion}
  There exists an inclusion $\Sym^{[d]}_{\cC} \Omega^p_X(\log D) \subseteq
  \Sym^{[d]} \Omega^p_X(\log \lceil D \rceil)$. \qed
\end{cor}

\begin{cor}[Reflexivity of $\cC$-differentials]\label{cor:reflexivity}
  The sheaf $\Sym^{[d]}_{\cC} \Omega^p_X(\log D)$ is a coherent, reflexive
  sheaf of $\sO_X$-modules, locally free wherever the pair $(X, \lceil D
  \rceil)$ is snc.
\end{cor}
\begin{proof}
  Corollary~\ref{cor:inclusion} represents $\sF := \Sym^{[d]}_{\cC}
  \Omega^p_X(\log D)$ as a subsheaf of the coherent sheaf $\sG := \Sym^{[d]}
  \Omega^p_X(\log \lceil D \rceil)$. We have also seen in
  Observation~\ref{obs:lfetc} that $\sF$ is locally free wherever that pair
  $(X, \lceil D \rceil)$ is snc. In particular it is locally free on an open
  subset $U \subseteq X$ whose complement has codimension $\geq 2$. In this
  setting, it follows from the classical extension theorem of coherent
  sheaves, \cite[I.~Thm.~9.4.7]{EGA1}, that there exists a coherent subsheaf
  $\sF' \subseteq \sG$ whose restriction to $U$ agrees with $\sF$. Since $\sF$
  is normal, and since the complement of $U$ is small, we have $\sF =
  (\sF')^{**}$.
\end{proof}

\subsubsection{Independence of $P(\tau, D_i)$ on choices, definition of defect
  divisors}

The independence of the numbers $P(\tau, D_i)$ on the choice of a particular
open set and an adapted morphism allows us to define a ``defect divisor'' that
measures additional fractional positivity along a $\cC$-differential. In
Section~\ref{ssec:defect}, we will extend this notion to sheaves of
differentials. Our starting point is the following Corollary, which summarizes
Observations~\ref{obs:Pindep} and \ref{obs:Pneg}.

\begin{cordefn}\label{cdef:P}
  Let $(X, D)$ be a $\cC$-pair and $\sigma$ a section of $\bigl( \Sym^{[d]}
  \Omega^p_X \bigr)(*\lceil D \rceil)$. Further, consider an open set $U
  \subseteq X$ and an adapted morphism $\gamma : V \to U$.  If $D_i \subseteq
  D$ is an irreducible component that intersects $U$ and has finite
  $\cC$-multiplicity and if $E \subset V$ is any divisor that dominates $D_i
  \cap U$, then $\gamma^{[*]}(\sigma)$ is a rational section of $\Sym^{[d]}
  \Omega^p_V(\log D_\gamma)$ whose pole order $P(\sigma, D_i)$
  along $E$ depends only on $\sigma$ and on the component $D_i \subseteq D$,
  but not on the choice of $U$, the morphism $\gamma$ or the particular
  divisor that dominates $D_i$.
  
  The section $\sigma$ is in $\Gamma\bigl( U,\, \Sym^{[d]} _{\cC}
  \Omega^p_X(\log D) \bigr)$ if and only if $P(\sigma, D_i) \leq 0$ for all
  components $D_i \subseteq D$ with finite $\cC$-multiplicity. \qed
\end{cordefn}

\begin{cor}\label{cor:indep}
  To check the conditions spelled out in Definition~\ref{def:orbifoldDiff}, it
  suffices to consider a single covering by open sets $(U_{\alpha})_{\alpha
    \in A}$ and for each $U_{\alpha}$ a single adapted morphism. \qed
\end{cor}

Using the numbers $P$ defined in \ref{cdef:P}, we define the defect divisor of
a $\cC$-differential.

\begin{defn}[Defect divisor of a $\cC$-differential]
  If $(X, D)$ is a $\cC$-pair, $D = \sum \frac{n_i-1}{n_i}\cdot D_i$ and
  $\sigma$ a section of $\Sym^{[d]} _{\cC} \Omega^p_X(\log D)$, consider the
  following $\Q$-Weil divisor,
  $$
  R(\sigma) := \sum_{\substack{D_i \subseteq D \text{ with } n_i < \infty}}
  \frac{-P(\sigma, D_i)}{n_i} \cdot D_i.
  $$
  We call $R(\sigma)$ the \emph{defect divisor} of the section $\sigma$.
\end{defn}

\begin{rem}\label{rem:univpropofdefect}
  The defect divisor $R(\sigma)$ is always effective. If two sections $\sigma$
  and $\tau$ of $\Sym^{[d]} _{\cC} \Omega^p_X(\log D)$ differ only by
  multiplication with a nowhere-vanishing function, their defect divisors
  $R(\sigma)$ and $R(\tau)$ agree.
\end{rem}

\subsubsection{The symmetric algebra of $\cC$-differentials}

The special form of the generators for $\Sym^{[d]}_{\cC}\bigl(\Omega^p_U(\log
D)\bigr)$ found in Observation~\ref{obs:lfetc} makes it possible to interpret
a tensor product of symmetric $\cC$-differentials as a
$\cC$-differential. More precisely, we obtain the following multiplication
morphisms.

\begin{cor}\label{cor:sympow}
  Since $\lfloor a \rfloor+\lfloor b \rfloor \leq \lfloor a+b \rfloor$ for any
  pair of numbers $a$ and $b$, the multiplication morphisms of symmetric
  differentials extend to multiplication morphisms of symmetric $\cC$-
  differentials. More precisely, given any two numbers $d_1$, $d_2 \in \mathbb
  N$, we obtain sheaf morphisms
  \begin{align*}
    \Sym^{[d_1]}_{\cC} \Omega^p_X(\log D) \otimes \Sym^{[d_2]}_{\cC}
    \Omega^p_X(\log D) & \to \Sym^{[d_1+d_2]}_{\cC} \Omega^p_X(\log D) \\
    \Sym^{d_1} \left( \Sym^{[d_2]}_{\cC} \Omega^p_X(\log D) \right) & \to
    \Sym^{[d_1 \cdot d_2]}_{\cC} \Omega^p_X(\log D)
  \end{align*}
  that agree outside of $\supp(D)$ with the usual multiplication maps.  \qed
\end{cor}

We obtain a symmetric algebra of $\cC$-differentials, which will allow us to
define a variant of the Kodaira-Iitaka dimension for sheaves of
$\cC$-differentials in Section~\ref{sec:SOD}.

\begin{cor}[Symmetric algebra of $\cC$-differentials]
  With the multiplication morphisms of Corollary~\ref{cor:sympow}, the direct
  sum $\bigoplus_{d \geq 0} \Sym^{[d]}_{\cC} \Omega^p_X(\log D)$ is a sheaf of
  $\sO_X$-algebras. \qed
\end{cor}

\subsubsection{Behavior under subadapted morphisms}

Equation~\eqref{eq:POs} immediately shows that the pull-back of
$\cC$-differentials under subadapted morphisms also become regular logarithmic
differentials.

\begin{cor}[Behaviour under subadapted morphisms]\label{cor:pbsadpt}
  Let $(X,D)$ be a $\cC$-pair and $\gamma : Y \to X$ a subadapted
  morphism. Similar to the setup of Definition~\ref{def:orbifoldDiff}, the
  natural pull-back morphism of differential forms extends to a morphism
  $\gamma^{[*]} \Sym^{[d]} _{\cC} \Omega^p_X(\log D) \to \Sym^{[d]}
  \Omega^p_Y(\log D_\gamma)$. \qed
\end{cor}

\subsubsection{A criterion for $\Sym^{[m]}_{\cC} \Omega^p_X(\log D)|_F$ to be anti-nef}

In Remark~\ref{rem:molitor} we considered the standard conormal sequence of
adapted differentials for a smooth curve $F \subset X$. The following
proposition gives a criterion for $\Sym^{[m]}_{\cC} \Omega^p_X(\log D)|_F$ to
be anti-nef, and will be an essential ingredient in the proof of
Theorem~\ref{thm:main}.

\begin{prop}\label{prop:gennbld}
  Let $F \subset X$ be a smooth curve and assume that the following holds.
  \begin{enumerate}
  \item\ilabel{il:caspar} The pair $\bigl(X, \lceil D \rceil \bigr)$ is snc
    along $F$, and $F$ intersects $\supp(D)$ transversely.
  \item The normal bundle $N_{F/X}$ is nef.
  \item The $\mathbb Q$-divisor $-(K_F+D|_F)$ is nef.
  \end{enumerate}
  If $m \in \mathbb N^+$ is any number and $1 \leq p \leq \dim X$, then
  $\Sym^{[m]}_{\cC} \Omega^p_X(\log D)|_F$ is anti-nef.
\end{prop}
\begin{proof}
  To start, observe that Condition~\iref{il:caspar} guarantees that
  $\Sym^{[m]}_{\cC} \Omega^p_X(\log D)|_F$ is locally free along $F$. Let $H
  \subset X$ be a general hyperplane, and $\gamma: Y \to X$ be a cyclic
  adapted cover with extra branching along $H$. Let $D_\gamma$ and $H_\gamma$
  be the divisors defined in Notation~\ref{not:dgamma} and
  \ref{not:adptweb}. Further, we consider the curve $\tilde F :=
  \gamma^{-1}(F)$. Observe that $\tilde F$ is smooth, that $Y$ is smooth along
  $\tilde F$, and that $\tilde F$ intersects $D_\gamma \cup H_\gamma$
  transversely.

  Since a sheaf is anti-nef if its pull-back under a finite map is anti-nef,
  it suffices to show that
  $$
  \gamma^*\bigl(\Sym^{[m]}_{\cC} \Omega^p_X(\log D)|_F \bigr) \subseteq
  \Sym^{m} \bigl( \Omega^p_Y(\log D_\gamma)_{\rm adpt}|_{\tilde F} \bigr)
  $$
  is anti-nef. Since subsheaves and tensor powers of anti-nef sheaves are
  anti-nef, it suffices to see that $\Omega^1_Y(\log D_\gamma)_{\rm
    adpt}|_{\tilde F}$ is anti-nef.  For that, recall the generalized conormal
  sequence of Remark~\ref{rem:molitor}, which presents $\Omega^1_Y(\log
  D_\gamma)_{\rm adpt}|_{\tilde F}$ as an extension of two bundles, both of
  which are anti-nef by assumption.
\end{proof}

\section{Sheaves of $\cC$-differentials and their Kodaira-Iitaka dimensions}
\label{sec:SOD}

Following \cite{Cam07} closely, we define a variant of the Kodaira-Iitaka
dimension for sheaves of $\cC$-differentials in Section~\ref{ssec:KIDS}, where
we also generalize the notion of ``special'' to $\cC$-pairs. In
Section~\ref{ssec:defect} we introduce the \emph{defect divisor} of a sheaf,
which helps in the computation of Kodaira-Iitaka dimensions.

Throughout the present Section~\ref{sec:SOD}, we consider a $\cC$-pair $(X,
D)$ as in Definition~\ref{def:orbifold} and let $\sF$ be a reflexive sheaf of
symmetric $\cC$-differentials with inclusion
$$
\iota: \sF \into \Sym^{[d]}_{\cC} \Omega^p_X(\log D).
$$
We assume that $\sF$ is saturated in $\Sym^{[d]}_{\cC} \Omega^p_X(\log D)$,
i.e., that the cokernel of $\iota$ is torsion free.

\subsection{Kodaira-Iitaka dimensions and special $\pmb \cC$-pairs}
\label{ssec:KIDS}

The usual definition of Kodaira-Iitaka dimension considers reflexive tensor
powers of a given reflexive sheaf of rank one. In our setup, where $\sF$ is a
reflexive sheaf of symmetric $\cC$-differentials, we aim to detect the
fractional positivity encoded in the $\cC$-pair by saturating the tensor
product in $\Sym^{[m\cdot d]}_{\cC} \Omega^p_X(\log D)$ before considering
sections. The following notation is useful in the description of the process.

\begin{notation}[\protect{$\cC$-product sheaves, cf.~\cite[Sect.~2.6]{Cam07}}]\label{not:orbifoldsaturation}
  Given a number $m \in \bN^+$, Corollary~\ref{cor:sympow} asserts that there
  exists a non-vanishing inclusion $\iota^m : \Sym^{[m]} \sF \into \Sym^{[m
    \cdot d]}_{\cC} \Omega^p_X(\log D)$.  Let $\Sym^{[m]}_{\cC} \sF$ be the
  saturation of the image, i.e., the kernel of the associated map
  $$
  \Sym^{[m \cdot d]}_{\cC} \Omega^p_X(\log D) \to \factor
  \coker(\iota^m).\textrm{tor}..
  $$
  We call $\Sym^{[m]}_{\cC} \sF$ the \emph{$\cC$-product of $\sF$}. There are
  inclusions
  $$
  \Sym^{[m]} \sF \into \Sym^{[m]}_{\cC} \sF \into \Sym^{[m \cdot d]}_{\cC}
  \Omega^p_X(\log D).
  $$  
\end{notation}

\begin{rem}\label{rem:osilf}
  The $\cC$-product $\Sym^{[m]}_{\cC} \sF$ is a saturated subsheaf of a
  reflexive sheaf and therefore itself reflexive, by
  Proposition~\ref{prop:satisrefl}.  If $\rank \sF = 1$, this implies that the
  restriction of $\Sym^{[m]}_{\cC} \sF$ to the smooth locus of $X$ is locally
  free, \cite[Lem.~1.1.15 on p.~154]{OSS}.
\end{rem}

\begin{defn}[\protect{$\cC$-Kodaira-Iitaka dimension, cf.~\cite[Sect.~2.7]{Cam07}}]\label{def:okodairaiitaka}
  If $X$ is projective and $\rank \sF = 1$, we consider the set
  $$
  M := \left\{ m \in \mathbb N\, \left|\, h^0\bigl(X,\, \Sym^{[m]}_{\cC} \sF
      \bigr) > 0 \right. \right\}.
  $$
  If $M = \emptyset$, we say that the sheaf $\sF$ has $\cC$-Kodaira-Iitaka
  dimension minus infinity, $\kappa_{\cC}(\sF)= - \infty$. Otherwise, by
  Remark~\ref{rem:osilf}, the restriction of $\Sym^{[m]}_{\cC} \sF$ to the
  smooth locus of $X$ is locally free, and we consider the natural rational
  mapping
  $$
  \phi_m : X \dasharrow \mathbb P\left( H^0\bigl(X,\, \Sym^{[m]}_{\cC} \sF
    \bigr)^\vee \right), \text{ for each } m \in M.
  $$
  Define the $\cC$-Kodaira-Iitaka dimension as
  $$
  \kappa_{\cC}(\sF) = \max_{m \in M} \left\{ \dim
    \overline{\phi_m(X)}\right\}.
  $$
\end{defn}

\begin{rem}\label{rem:kcandk}
  If $D=\emptyset$, or if $(X,D)$ is a logarithmic pair, it is clear from the
  construction and from the saturatedness assumption that $\Sym_{\cC}^{[m]}
  \sF \cong \Sym^{[m]} \sF$ for all $m$, and that the $\cC$-Kodaira-Iitaka
  dimension of $\sF$ therefore equals the regular Kodaira-Iitaka dimension,
  $\kappa_{\cC}(\sF) = \kappa(\sF)$.
\end{rem}

\begin{rem}[Invariance of $\kappa_{\cC}$ under $\cC$-products]\label{rem:iokuop}
  Using Remark~\ref{rem:kcandk}, standard arguments show that if $X$ is
  projective, then $\kappa_{\cC}(\sF) = \kappa_{\cC}\bigl( \Sym^{[m]}_{\cC}
  \sF \bigr)$ for all positive $m$.
\end{rem}

\begin{warning}
  Unlike the standard Kodaira-Iitaka dimension, the $\cC$-Kodaira-Iitaka
  dimension is defined only for subsheaves of $\Sym^{[d]}_{\cC}
  \Omega^p_X(\log D)$. Its value is generally not an invariant of the sheaf
  alone and will often depend on the embedding.
\end{warning}

Using the $\cC$-Kodaira-Iitaka dimension instead of the standard definition,
we have the following immediate generalization of
Definition~\ref{def:speciallog}, which agrees with the old definition if
$(X,D)$ is a logarithmic pair.

\begin{defn}[\protect{Special $\cC$-pairs, cf. \cite[Def.~4.18 and Thm.~7.5]{Cam07}}]\label{def:specialingen}
  A $\cC$-pair $(X, D)$ is \emph{special} if $\kappa_{\cC}(\sF)<p$ for any
  number $1 \leq p \leq \dim X$ and any saturated rank-one sheaf $\sF
  \subseteq \Sym^{[1]}_{\cC} \Omega^p_X(\log D)$.
\end{defn}

\subsection{Defect divisors for sheaves of $\pmb \cC$-differentials}
\label{ssec:defect}

If $\rank \sF = 1$, then $\sF|_{X_{\reg}}$ is locally free. If $U_1$ and $U_2
\subseteq X_{\reg}$ are open subsets of the smooth locus and if $\sigma_i \in
\GA \bigl(U_i,\, \sF \bigr)$ are generators of $\sF|_{U_i}$, this implies that
$\sigma_1|_{U_1 \cap U_2}$ and $\sigma_2|_{U_1 \cap U_2}$ differ only by
multiplication with a nowhere-vanishing function. In particular,
Remark~\ref{rem:univpropofdefect} asserts that the defect divisors
$R(\sigma_1)$ and $R(\sigma_2)$ agree on the overlap $U_1 \cap U_2$.  The
following definition therefore makes sense.

\begin{defn}[Defect divisor and $\cC$-divisor class of a sheaf of differentials]\label{def:defect}
  If $\rank \sF = 1$, let $R_{\sF}$ be the unique $\Q$-Weil divisor on $X$
  such that for any open set $U \subseteq X_{\reg}$, and any generator $\sigma
  \in \Gamma \bigl(U,\, \sF \bigr)$, we have $R_{\sF} \cap U = R(\sigma)$. We
  call $R_{\sF}$ the \emph{defect divisor} of the sheaf $\sF$.

  Recall that there exists, up to linear equivalence, a unique Weil divisor
  $W$ such that $\sF = \sO_X(W)$. Let $\Div(\sF) \in \Cl(X)$ be the associated
  element of the divisor class group. If $X$ is $\Q$-factorial, we define the
  \emph{$\cC$-divisor class of the sheaf $\sF$}, written $\Div_{\cC}(\sF)$, as
  the $\Q$-linear equivalence class given by $\Div_{\cC}(\sF) := \Div(\sF) +
  R_{\sF}$.
\end{defn}

\begin{rem}[Pull-back of defect divisor under adapted morphisms]\label{rem:pbddbi}
  In the setup of Definition~\ref{def:defect}, if $U \subseteq X$ is any open
  set and $\gamma: V \to U$ any adapted morphism, it is clear from the
  definition that $\gamma^*(R_{\sF})$ is an integral Weil divisor on $V$.
\end{rem}

\begin{rem}[Characterization of the defect divisor]\label{rem:charDefect}
  In the setup of Definition~\ref{def:defect}, if $U \subseteq X$ is any open
  set and $\gamma: V \to U$ is any adapted morphism,
  Definition~\ref{def:orbifoldDiff} of $\cC$-differentials asserts that there
  exists an inclusion
  $$
  i : \gamma^{[*]}(\sF) \into \Sym^{[d]} \Omega^p_V ( \log D_{\gamma} )
  $$
  which factors into a sequence of inclusions,
  \begin{equation}\label{eq:factincl}
    \xymatrix{ \gamma^{[*]}(\sF) \ar[r] & \left( \gamma^*(\sF)
        \otimes \sO_V \bigl( \gamma^* R_{\sF} \bigr) \right)^{**}
      \ar[r]^(0.54){j} & \Sym^{[d]} \Omega^p_V ( \log D_{\gamma} ), }
  \end{equation}
  where the cokernel of $j$ is torsion free in codimension one. The defect
  divisor $R_{\sF}$ is uniquely determined by this property.
\end{rem}

We next show that the defect divisor behaves nicely under $\cC$-products.

\begin{prop}[Behaviour under $\cC$-products]\label{prop:prod}
  In the setup of Definition~\ref{def:defect}, if $m \in \bN^+$ is any number,
  we have
  \begin{align}
    \label{eq:Rmul1} \Sym^{[m]}_{\cC} \sF & = \left( \sF^{[m]} \otimes
      \sO_X\bigl( \lfloor m \cdot R_{\sF} \rfloor \bigr) \right)^{**} \\
   \label{eq:Rmul2} R_{\Sym^{[m]}_{\cC} \sF} & = \underbrace{m \cdot
     R_{\sF} - \lfloor m \cdot R_{\sF} \rfloor}_{=: Q}
  \end{align}
\end{prop}
\begin{proof}
  Let $U \subseteq X$ be any open set, and $\gamma: V \to U$ any finite
  adapted morphism. Then there exist open sets $U^\circ \subseteq U \cap
  X_{\reg}$ and $V^\circ \subseteq \gamma^{-1}(U^\circ)$ with $\codim_U
  U\setminus U^\circ = \codim_V V\setminus V^\circ \geq 2$ such that both the
  sheaf $\Omega^p_V(\log D_{\gamma})$ and the cokernel of the injection
  $$
  j: \gamma^*(\sF) \otimes \sO_{V^\circ}\bigl(\gamma^* R_{\sF} \bigr) \into
  \Sym^d \Omega^p_V(\log D_{\gamma})
  $$
  are locally free on $V^\circ$.  Taking $m^{\rm{th}}$ symmetric products, the
  inclusion $j$ yields an inclusion of sheaves on $V^\circ$,
  \begin{equation}\label{eq:imequ}
    j^m: \underbrace{\Sym^m \left( \gamma^*(\sF) \otimes
      \sO_{V^\circ}\bigl(\gamma^* R_{\sF} \bigr) \right)}_{=:\, \sA} \into
    \Sym^{m\cdot d} \Omega^p_V(\log D_{\gamma}),
  \end{equation}
  with locally free cokernel. On $V^\circ$ and $U^\circ$, respectively, the
  domain $\sA$ of the map $j^m$ can then be written as follows.
  \begin{equation}\label{eq:inclll}
    \begin{split}
      \sA & \cong \Sym^m \left( \gamma^*(\sF) \right) \otimes \sO_{V^\circ}\bigl(m \cdot \gamma^* R_{\sF} \bigr) \\
      & \cong \gamma^* \bigl( \sF^m \bigr) \otimes \sO_{V^\circ}\bigl( \gamma^* (m \cdot R_{\sF}) \bigr) \\
      & \cong \gamma^* \bigl( \sF^m \otimes \sO_{U^\circ}(\lfloor m \cdot R_{\sF} \rfloor) \bigr) \otimes \sO_{V^\circ}\bigl( \gamma^* Q \bigr)
  \end{split}
  \end{equation}
  where $Q$ is the $\bQ$-divisor defined in~\eqref{eq:Rmul2} above. Since $Q$
  is effective, Inclusion~\eqref{eq:imequ} gives an inclusion of locally free
  sheaves on $U^\circ$,
  $$
  \sF^m \otimes \sO_{U^\circ}(\lfloor m \cdot R_{\sF} \rfloor) \subseteq
  \Sym^{[m]}_{\cC} \sF.
  $$
  In particular, there exists an effective Cartier divisor $P$ such that
  $$
  \sF^m \otimes \sO_{U^\circ}(\lfloor m \cdot R_{\sF} \rfloor) \otimes
  \sO_{U^\circ}(P) = \Sym^{[m]}_{\cC} \sF.
  $$
  Thus
  $$
  \gamma^* \bigl(\sF^m \otimes \sO_{U^\circ}(\lfloor m \cdot R_{\sF} \rfloor)
  \otimes \sO_{U^\circ}(P) \bigr) \subseteq \Sym^{m\cdot d} \Omega^p_V(\log
  D_{\gamma}).
  $$
  But since the cokernel of $j^m$ is locally free, Equation~\eqref{eq:inclll}
  implies that $\gamma^*(P) \leq \gamma^*(Q)$. Since $\lfloor Q \rfloor = 0$,
  this is possible if and only if $P = 0$. This shows
  Assertion~\eqref{eq:Rmul1}.  Assertion~\eqref{eq:Rmul2} then follows from
  the characterization of the defect divisor given in
  Remark~\ref{rem:charDefect}, Equation~\eqref{eq:inclll} and again from the
  fact that the cokernel of $j^m$ is locally free.
\end{proof}

As an immediate corollary, we can relate the $\cC$-Kodaira-Iitaka dimension of
a rank one subsheaf of $\Sym^{[d]}_{\cC} \Omega^p_X(\log D)$ to the standard
Kodaira-Iitaka dimension.

\begin{cor}\label{cor:realizationofkappa}
  In the setup of Definition~\ref{def:defect}, if $m \in \bN^+$ is any number,
  and $\gamma: Y \to X$ any adapted morphism, then there exists a
  sequence of inclusions as follows:
  $$
  \xymatrix{ \gamma^{[*]} \bigl( \Sym^{[m]}_{\cC} \sF \bigr) \ar[r] &
    \Sym^{[m]}\left( \gamma^*(\sF) \otimes \sO_Y \bigl( \gamma^* R_{\sF}
      \bigr) \right) \ar[r] & \Sym^{[m\cdot d]} \Omega^p_Y ( \log D_{\gamma}
    ).}
  $$
  If $X$ is projective and if $\gamma$ is proper, then $\kappa_{\cC}(\sF)
  \leq \kappa \left( \bigl( \gamma^*(\sF) \otimes \sO_Y ( \gamma^* R_{\sF})
    \bigr)^{**} \right)$.
\end{cor}
\begin{proof}
  Substitute Equations~\eqref{eq:Rmul1} and \eqref{eq:Rmul2} of
  Proposition~\ref{prop:prod} into the Sequence~\eqref{eq:factincl} to obtain
  the sequence of inclusions.  The inequality of Kodaira-Iitaka dimensions
  follows immediately from the definition of $\kappa_{\cC}$ and from the first
  inclusion.
\end{proof}

The following fact is another immediate consequence of
Proposition~\ref{prop:prod} and of Remark~\ref{rem:iokuop}.  

\begin{cor} \label{cor:kappaofample}
  If $X$ is projective, and if $m \in \bN^+$ is any number such that $m\cdot
  R_{\sF}$ is an integral divisor, then $\kappa_{\cC}(\sF) = \kappa \bigl(
  m\cdot \Div_{\cC}(\sF) \bigr)$. \qed
\end{cor}

\section{The $\cC$-pair associated with a fibration}
\label{sec:ofibr}

If $(Y, D)$ is a logarithmic pair, and $\pi : Y \to Z$ a fibration, we aim to
describe the maximal divisor $\Delta$ on $Z$ such that $\cC$-differentials of
the pair $(Z, \Delta)$ pull back to logarithmic differentials on $(Y,
D)$. Once $\Delta$ is found, we will see in Proposition~\ref{prop:dfhg} that
any section in $\Sym^{[m]} \Omega_Y^p(\log D)$ which generically comes from
$Z$ is really the pull-back of a globally defined $\cC$-differential from
downstairs. The construction of $\Delta$ is originally found in slightly
higher generality in \cite[Sec.~3.1]{Cam07}, where the $\cC$-pair $(Z,
\Delta)$ is called the \emph{base orbifolde} of the fibration. This section
contains a short review of the construction, as well as detailed and
self-contained proofs of all results required later.

In order to keep the technical apparatus reasonably small, we restrict
ourselves to logarithmic pairs in this section, which is the case we need to
handle in the proof of Theorem~\ref{thm:main}. The definitions and results of
this section can be generalized in a straightforward manner to the case of
arbitrary $\cC$-pairs.

\subsection{Definition of the $\pmb \cC$-base}

The following setup is maintained throughout the present
Section~\ref{sec:ofibr}.

\begin{setup}\label{setup:fibr}
  Let $(Y, D)$ be a logarithmic pair, and $\pi : Y \to Z$ a proper, surjective
  morphism with connected fibers to a normal space.
\end{setup}

\begin{notation}[Log discriminant locus]\label{not:logdiscr}
  The \emph{log discriminant locus} $S \subset Z$ is the smallest closed set
  $S$ such that $\pi$ is smooth away from $S$, and such that for any point $z
  \in Z \setminus S$, the fiber $Y_z := \pi^{-1}(z)$ is not contained in $D$,
  and the scheme-theoretic intersection $Y_z \cap D$ is an snc divisor in
  $Y_z$. We decompose
  $$
  S = S_{\div} \cup S_{\sml},
  $$
  where $S_{\div}$ is a divisor, and $\codim_Z S_{\sml} \geq 2$. The divisor
  $S_{\div}$ is always understood to be reduced.
\end{notation}

\begin{construction-defn}[\protect{$\cC$-base of the fibration, cf.~\cite[Def 3.2]{Cam07}}]\label{cons:cdefn}
  Let $S_{\div} = \bigcup_i \Delta_i$ be the decomposition into irreducible
  components. We aim to attach multiplicities $a_i \in \mathbb Q^{\geq 0}$ to
  the components $\Delta_i$, in order to define a $\cC$-divisor $\Delta :=
  \sum_i a_i\cdot \Delta_i$.

  To this end, let $Z^\circ \subseteq Z$ be the maximal open subset such that
  $\pi$ is equidimensional over $Z^\circ$. Set $Y^\circ := \pi^{-1}(Z^\circ)$,
  and observe that all components $\Delta_i$ intersect $Z^\circ$
  non-trivially. In particular, none of the divisors $\Delta_i^\circ :=
  \Delta_i \cap Z^\circ$ is empty. Given one component $\Delta_i$, the
  preimage $\pi^{-1}(\Delta_i^\circ)$ has support of pure codimension one in
  $Y^\circ$, with decomposition into irreducible components
  $$
  \supp \bigl( \pi^{-1}(\Delta_i^\circ) \bigl) = \bigcup_j E_{i,j}^\circ.
  $$
  If for the given index $i$, all $E_{i,j}^\circ$ are contained in $D$, set
  $a_i := 1$. Otherwise, set
  $$
  b_i := \min \{ \text{multiplicity of } E_{i,j}^\circ \text{ in }
  \pi^{-1}(\Delta_i^\circ) \,|\, E_{i,j}^\circ \not \subset D\} \quad
  \text{and} \quad a_i := \frac{b_i-1}{b_i}.
  $$
  We obtain a divisor $\Delta := \sum_i a_i \cdot \Delta_i$ with $\supp(\Delta)
  \subseteq S_{\div}$. We call the $\cC$-pair $(Z, \Delta)$ the
  \emph{$\cC$-base of the fibration $\pi$}.
\end{construction-defn}

The notion of the $\cC$-base of a fibration is not very useful unless the
fibration and the spaces have further properties, cf.~Remark~\ref{srem:pbcd}
below. We will therefore maintain the following assumptions throughout the
remainder of the current Section~\ref{sec:ofibr}.

\begin{assumption}\label{ass:fibration}
  In Setup~\ref{setup:fibr}, assume additionally that the following holds.
  \begin{enumerate}
  \item The pair $(Y, D)$ is snc. In particular, $Y$ is smooth.
  \item The pair $(Z, S_{\div})$ is snc. In particular, $Z$ is smooth.
  \item Every irreducible divisor $E \subset Y$ with $\codim_Z \pi(E) \geq 2$
    is contained in $D$.
  \end{enumerate}
\end{assumption}

\subsection{The pull-back map for $\pmb \cC$-differentials and sheaves}\label{ssec:pbbf}

Assumptions~\ref{ass:fibration} guarantee that $\cC$-differentials on $(Z,
\Delta)$ can be pulled back to logarithmic differentials on $(Y, D)$. In fact,
a slightly stronger statement holds.

\begin{prop}\label{prop:spcdwn}
  Under the Assumptions~\ref{ass:fibration}, decompose the divisor $D = D^h
  \cup D^v$ into the ``horizontal'' components $D^h$ that dominate $Z$ and the
  ``vertical'' components $D^v$ that do not. With $\Delta$ as in
  Construction~\ref{cons:cdefn}, the pull-back morphism of differentials
  extends to a map\footnote{Since $\Omega_Y^p(\log D^v)$ is locally free, we
    could write $\Sym^m \Omega_Y^p(\log D^v)$ instead of the more complicated
    $\Sym^{[m]} \Omega_Y^p(\log D^v)$. We have chosen to keep the square
    brackets throughout in order to be consistent with the notation used in
    the remainder of this paper.}
  \begin{equation}\label{eq:pbob2}
    d\pi^m : \pi^{[*]} \bigl( \Sym^{[m]}_{\cC} \Omega_Z^p(\log \Delta)\bigr) \to
    \Sym^{[m]} \Omega_Y^p(\log D^v)
  \end{equation}
  for all numbers $m$ and $p$.
\end{prop}
\begin{subrem}\label{srem:pbcd}
  For Proposition~\ref{prop:spcdwn}, it is essential to assume that the pair
  $(Z, S_{\div})$ is snc. For an instructive example, let $Z$ be a singular
  space, $\pi: Y \to Z$ a log desingulariazion of $Z$, let $D$ be the
  $\pi$-exceptional locus, and take $m=1$ and $p = \dim Z$. In this setting,
  the assertion of Proposition~\ref{prop:spcdwn} holds if and only if the pair
  $(Z, \emptyset)$ is log canonical ---this is actually the definition of log
  canonicity. We refer to \cite{GKK08, GKKP10} for more general results in
  this context.
\end{subrem}
\begin{proof}[Proof of Proposition~\ref{prop:spcdwn}]
  Let $U \subseteq Z$ be an open set and let $\sigma \in \GA\bigl( U,\,
  \Sym^{[m]}_{\cC} \Omega_Z^p(\log \Delta) \bigr)$ be any section. Its
  pull-back $\pi^{[*]}(\sigma)$ then gives a rational section of the sheaf
  $\Sym^{[m]} \Omega_Y^p(\log D^v)$, possibly with poles along the
  codimension-one components of $\pi^{-1}(S)$. We need to show that
  $\pi^{[*]}(\sigma)$ does in fact not have any poles. To this end, let $E
  \subseteq \pi^{-1}(S)$ be any irreducible component with $\codim_Y E =
  1$. We will show $\pi^{[*]}(\sigma)$ does not have any poles along $E$.

  If $E \subseteq D^v$, note that
  $$
  \sigma \in \Gamma \bigl( U,\, \Sym^{[m]}_{\cC} \Omega_Z^p(\log \Delta)
  \bigr) \subseteq \Gamma \bigl( U,\, \Sym^{[m]} \Omega_Z^p(\log S_{\div})
  \bigr).
  $$
  Away from the small set in $Y^\circ$ where $\bigl(Y, \supp \pi^{-1}(S_{\div})
  \bigr)$ is not snc, the usual pull-back morphism for logarithmic
  differentials, $\pi^* \bigl( \Omega_Z^p(\log S_{\div}) \bigr) \to \Omega_Y^p
  \bigl(\log \supp \pi^{-1}(S_{\div}) \bigr)$ shows that $\pi^{[*]}(\sigma)$
  has at most logarithmic poles along $E$. In particular, $\pi^{[*]}(\sigma)$
  does not have any poles along $E$ as a section of $\Sym^{[m]}
  \Omega_Y^p(\log D^v)$.
  
  It remains to consider the case where $E \not \subset D^v$. In this case,
  Assumptions~\ref{ass:fibration} guarantee that $E$ dominates a component of
  $S_{\div}$. For simplicity of notation, we may remove from $Z$ all other
  irreducible components of $S$, and also the small set where $\pi$ is not
  equidimensional. We can then assume without loss of generality that $S =
  \pi(E)$, and that the restricted morphism $\pi|_{Y \setminus D}$ is
  surjective and equidimensional. By construction of $\Delta$, the morphism
  $\pi|_{Y \setminus D}$ is then subadapted, in the sense of
  Definition~\ref{def:adaptedmorphism}. In particular,
  Corollary~\ref{cor:pbsadpt} shows that $(\pi|_{Y \setminus
    D})^{[*]}(\sigma)$ is a section of $\Sym^{[m]} \Omega^p_{Y\setminus D}$
  without any poles along $E \cap (Y \setminus D)$.
\end{proof}

As an immediate corollary we see that the $\cC$-base of the fibration $\pi$ is
special if the logarithmic pair $(Y,D)$ is special.

\begin{cor}\label{cor:baseofspecialisspecial}
  Under the Assumptions~\ref{ass:fibration}, if the logarithmic pair $(Y,D)$
  is special in the sense of Definition~\ref{def:speciallog}, then the
  $\cC$-pair $(Z, \Delta)$ is special in the generalized sense of
  Definition~\ref{def:specialingen}. \qed
\end{cor}

\subsection{The push-forward map for $\pmb \cC$-differentials and sheaves}\label{ssec:pfbf}

To properly formulate the assumption that a section in $\Sym^{[m]}
\Omega_Y^p(\log D)$ comes from $Z$ ``generically'', consider the sheaf $\sB
\subseteq \Sym^{[m]} \Omega_Y^p(\log D)$, defined to be the saturation of the
image of the map $d\pi^m$ introduced in~\eqref{eq:pbob2}. The following
proposition then says that any section in $\sB$ comes from a globally defined
section on $Z$.

\begin{prop}\label{prop:dfhg}
  Under the Assumptions~\ref{ass:fibration}, if $\sB \subseteq \Sym^{[m]}
  \Omega_Y^p(\log D)$ is the saturation of the image of the map $d\pi^m$
  introduced in~\eqref{eq:pbob2}, then the natural injection
  $$
  \iota: \Sym^{[m]}_{\cC} \Omega_Z^p(\log \Delta) \to \pi_*(\sB)
  $$
  is isomorphic for all numbers $m$ and $p$.
\end{prop}
\begin{subrem}\label{rem:bisgenpb}
  Since the morphism $\pi$ is log smooth over $Z^\circ := Z \setminus S$, the
  standard sequence of logarithmic differentials on the preimage set $Y^\circ
  := \pi^{-1}(Z^\circ)$,
  $$
  \xymatrix{ 0 \ar[r] & \pi^*\bigl(\Omega^1_Z\bigr)|_{Y^\circ}
    \ar[r]^(.45){d\pi|_{Y^\circ}} & \Omega^1_Y(\log D)|_{Y^\circ} \ar[r] &
    \Omega^1_{Y/Z}|_{Y^\circ}\otimes \sO_{Y^\circ}(D) \ar[r] & 0, }
  $$
  shows that the cokernel of $d\pi|_{Y^\circ}$ is torsion free on $Y^\circ$
  and that the image of $d\pi|_{Y^\circ}$ is saturated in $\Omega^1_Y(\log
  D)|_{Y^\circ}$. By \cite[II,~Ex.~5.16]{Ha77}, the same holds for $p$-forms
  and their symmetric products. By Proposition~\ref{prop:satandsym}, the
  sheaves $\sB$ and $\pi^{[*]} \bigl(\Sym^{[m]}_{\cC} \Omega_Z^p(\log
  \Delta)\bigr)$ therefore agree along $Y^\circ$.
\end{subrem}

\begin{proof}
  Since $\Sym^{[m]}_{\cC} \Omega_Z^p(\log \Delta)$ is reflexive and
  $\pi_*(\sB)$ is the push-forward of a torsion-free sheaf, hence torsion
  free, it suffices to prove surjectivity of $\iota$ away from any given small
  set. We can therefore assume without loss of generality throughout the proof
  that $\pi$ is equidimensional and that $S_{\sml} = \emptyset$.

  Let $U \subseteq Z$ be any open set and let $\sigma \in \Gamma\bigl( U,\,
  \pi_*(\sB)\bigr)$ be any section.  By Remark~\ref{rem:bisgenpb}, the sheaves
  $\sB$ and $\pi^* \bigl(\Sym^{[m]}_{\cC} \Omega_Z^p(\log \Delta)\bigr)$ agree
  along $\pi^{-1}(U \cap Z^\circ)$. Since $\pi_*(\sO_Y) = \sO_Z$, the section
  $\sigma$ therefore induces a section
  $$
  \sigma' \in \Gamma\bigl( U \cap Z^\circ ,\, \Sym^{[m]}_{\cC} \Omega_Z^p(\log
  \Delta) \bigr).
  $$
  The sections $\sigma$ and $\sigma'$ define saturated subsheaves
  $$
  \sA \subseteq \sB|_{\pi^{-1}(U)} \text{\quad and \quad} \sA' \subseteq
  \Sym^{[m]}_{\cC} \Omega_Z^p(\log \Delta)|_U,
  $$
  together with an inclusion $d\pi^m : \pi^*(\sA') \to \sA$. We need to show
  that the obvious injective map
  \begin{equation}\label{eq:dpionh0}
    d\pi^m: \Gamma\bigl(U,\, \sA' \bigr) \to \Gamma\bigl(\pi^{-1}(U),\, \sA\bigr)    
  \end{equation}
  is surjective.

  As in Construction~\ref{cons:cdefn}, decompose $S_{\div} = \cup \Delta_i$
  into irreducible components. For any given index $i$, let $E_{i,j} \subset
  Y$ be those divisors that dominate $\Delta_i$.  Observe that $\Sym^{[m]}
  \Omega_Y^p(\log D)$ and $\Sym^{[m]}_{\cC} \Omega_Z^p(\log \Delta)$ are both
  locally free. In particular, the saturated subsheaves $\sA$ and $\sA'$ are
  reflexive of rank one, hence invertible, \cite[Lem.~1.1.15]{OSS}, and there
  exist non-negative numbers $c_{i,j}$ such that
  $$
  \sA \cong \pi^*(\sA') \otimes \sO_Y \bigl({\textstyle \sum} c_{i,j}E_{i,j} \bigr).
  $$
  With this notation, surjectivity of~\eqref{eq:dpionh0} is an immediate
  consequence of the following claim.
  \begin{subclaim}\label{sclm:mlf}
    For any index $i$ with $\Delta_i \cap U \not = \emptyset$, there exists an
    index $j$ such that $E_{i,j}$ appears in $\pi^*(S_{\div})$ with
    multiplicity strictly larger than $c_{i,j}$.
  \end{subclaim}

  \noindent \emph{Application of Claim~\ref{sclm:mlf}.} Assume that
  Claim~\ref{sclm:mlf} holds true. We can view $\sigma'$ as a
  $\cC$-differential with poles along the $\Delta_i$,
  $$
  \sigma' \in \Gamma \left( U,\, \bigl(\Sym^{[m]}_{\cC} \Omega_Z^p(\log \Delta)
  \bigr)\otimes \sO_Z(m_i \Delta_i) \right).
  $$
  We need to show that all numbers $m_i$ are zero. Observe that the section
  $\sigma$ can be seen as a rational section in $\pi^*(\sA')$ whose pole order
  along any component $E_{i,j}$ is at least $m_i$ times the multiplicity of
  $E_{i,j}$ in $\pi^*(\Delta_i)$. With Claim~\ref{sclm:mlf}, this is possible
  if and only if $m_i = 0$ for all indices $i$. In particular, $\sigma$ lies
  in the image of the map~\eqref{eq:dpionh0}. Proposition~\ref{prop:dfhg} is
  thus shown once Claim~\ref{sclm:mlf} is established.

  \medskip

  \noindent \emph{Proof of Claim~\ref{sclm:mlf}.} To prove
  Claim~\ref{sclm:mlf}, let any index $i$ be given.

  If $a_i = 1$, let $j$ be any other index. By definition of $a_i$, the
  divisor $E_{i,j}$ is then contained in $D$. Let $y \in Y$ be a general point
  of $E_{i,j}$ and set $z := \pi(y)$.  Claim~\ref{sclm:mlf} then reduces to
  the standard fact that near $z$ and $y$, respectively, the pull-back of a
  local generator of $\Omega^p_Z(\log \Delta_i)$ gives a non-vanishing section
  in $\Omega^p_Y(\log E_{i,j})$. It follows that $c_{i,j} = 0$ for all $j$,
  proving Claim~\ref{sclm:mlf} in this case.

  If $a_i < 1$, then there exists an index $j$ such that $E_{i,j} \not \subset
  D$, such that $b_i$ is the multiplicity of $E_{i,j}$ in $\pi^*(\Delta_i)$,
  and $a_i = \frac{b_i-1}{b_i}$. As above, let $y \in Y$ be a general point of
  $E_{i,j}$ and set $z := \pi(y)$.  Thus, if we set
  $$
  U^\circ := U \setminus \bigcup_{i' \not = i} \Delta_{i'} \text{\quad and
    \quad} V^\circ := \pi^{-1}(U^\circ) \setminus \bigcup_{j' \not = j}
  E_{i,j'},
  $$
  then $y \in V^\circ$, $z \in U^\circ$, and the morphism $\pi^\circ :=
  \pi|_{V^\circ}$ is adapted. Now, if the claim was false and $b_i \leq
  c_{i,j}$, we obtain a morphism
  $$
  (\pi^\circ)^*(\sA' \otimes \sO_{U^\circ}(\Delta_i)) \to \sA|_{V^\circ}
  \subseteq \Sym^{[m]} \Omega^p_{V^\circ}(\log D).
  $$
  By Definition~\ref{def:orbifoldDiff} of $\cC$-differentials and by
  Corollary~\ref{cor:indep}, this says that $\sA' \otimes
  \sO_{U^\circ}(\Delta_i)$ is a subsheaf of $\Sym^{[m]}_{\cC}
  \Omega_{U^\circ}^p(\log \Delta)$, contradicting the assumption that $\sA'$
  is saturated in $\Sym^{[m]}_{\cC} \Omega_{U^\circ}^p(\log \Delta)$.
\end{proof}

We end this section with a discussion of push-forward properties of subsheaves
of $\sB$. Given a saturated subsheaf $\sA \subseteq \sB$ of rank one on $Y$,
with non-negative Kodaira-Iitaka dimension, we can construct a reflexive rank
one subsheaf $\sA_Z \subseteq \Sym^{[m]}_{\cC} \Omega^p_Z(\log \Delta)$ on the
base of the fibration, whose $\cC$-Kodaira-Iitaka dimension agrees with the
standard Kodaira-Iitaka dimension of $\sA$.  This sheaf will be used in the
proof of Theorem \ref{thm:main}.

\begin{cor}\label{cor:pdvzs}
  In the setup of Proposition~\ref{prop:dfhg}, let $\sA \subseteq \sB$ be a
  saturated subsheaf of rank one with $\kappa(\sA) \geq 0$.  Then there exists
  a saturated, reflexive subsheaf $\sA_Z \subseteq \Sym^{[m]}_{\cC}
  \Omega^p_Z(\log \Delta)$ of rank one such that $d\pi^m \bigl( \pi^*(\sA_Z)
  \bigr) \subseteq \sA$ and $\kappa_{\cC}(\sA_Z) = \kappa(\sA)$.
\end{cor}
\begin{proof}
  If $F \subset Y$ is a general $\pi$-fiber, Remark~\ref{rem:bisgenpb} implies
  that the restriction $\sB|_F$ is trivial. Since a tensor product of the
  restriction $\sA|_F \subseteq \sB|_F$ has a non-trivial section by
  assumption, this implies that $\sA|_F$ is also trivial. In particular, the
  sheaf $\pi_*(\sA)$ is generically of rank one. Consider the inclusion
  $$
  \pi_*(\sA) \subseteq \pi_*(\sB) \cong \Sym^{[m]}_{\cC} \Omega_Z^p(\log
  \Delta)
  $$
  and let $\sA_Z$ be the saturation of $\pi_*(\sA)$ in $\Sym^{[m]}_{\cC}
  \Omega_Z^p(\log \Delta)$. It is clear that $d\pi^m \bigl( \pi^*(\sA_Z)
  \bigr) \subseteq \sA$ holds generically, and since $\sA$ is saturated, this
  inclusion will hold everywhere.

  It remains to show that $\kappa_{\cC}(\sA_Z) = \kappa(\sA)$. The inequality
  $\kappa_{\cC}(\sA_Z) \leq \kappa(\sA)$ is clear. To prove that
  $\kappa_{\cC}(\sA_Z) \geq \kappa(\sA)$, note that if $m'$ is any number, if
  $\sB'$ is the saturation of the image
  $$
  d\pi^{m \cdot m'} : \pi^* \bigl( \Sym^{[m \cdot m']}_{\cC} \Omega_Z^p(\log
  \Delta)\bigr) \to \Sym^{[m \cdot m']} \Omega_Y^p(\log D^v)
  $$
  and $\sigma \in \Gamma \bigl( Y,\, \sA^{\otimes m'}\bigr)$ any section, then
  the inclusion $\sA^{\otimes m'} \subseteq \sB'$ shows that $\sigma$ induces
  a section $\sigma' \in \Gamma \bigl(Z,\, \Sym^{[m]}_{\cC} \Omega_Z^p(\log
  \Delta)\bigr)$ which, away from $S_{\div}$, lies in $\sA_Z^{[m']} \subseteq
  \Sym^{[m]}_{\cC} \Omega_Z^p(\log \Delta)$. It follows that $\sigma'$ is a
  section in the saturation of $\sA_Z^{[m']}$ which, by definition, is
  precisely $\Sym^{[m']}_{\cC} \sA_Z$. In summary, we obtain an injection
  $$
  \Gamma \bigl( Y,\, \sA^{\otimes m'}\bigr) \to \Gamma \bigl( Z,\,
  \Sym^{[m']}_{\cC} \sA_Z\bigr).
  $$
  This shows the equality of Kodaira-Iitaka dimensions.
\end{proof}

\part{FRACTIONAL POSITIVITY}\label{P2}

\section{The slope filtration for $\cC$-differentials}
\label{sec:slope}

The results of the following two sections are new to the best of our
knowledge.  In this section we discuss a weak variant of the Harder-Narasimhan
filtration that works on sheaves of $\cC$-differentials and takes the extra
fractional positivity of these sheaves into account.

If $X$ is a normal polarized variety, $\sF$ a reflexive sheaf with slope
$\mu(\sF) \leq 0$ and $\sA \subset \sF$ a subsheaf with positive slope, it is
clear that the maximally destabilizing subsheaf of $\sF$ is a proper subsheaf
of positive slope. In particular, there exists a number $p < \rank \sF$, and a
rank-one subsheaf $\sB \subset \bigwedge^{[p]} \sF$ that is likewise of
positive slope $\mu(\sB) > 0$. The following proposition gives a similar, but
slightly stronger result when $\sF$ is replaced with the sheaf of
$\cC$-differentials.

\begin{prop}\label{prop:HNF}
  Let $(X, D)$ be a $\cC$-pair of dimension $n$, as in
  Definition~\ref{def:orbifold}. Assume that $X$ is projective and
  $\bQ$-factorial, and let $A$ be an ample Cartier divisor. If $(K_X +
  D).A^{n-1} \leq 0$ and if there exists a number $m$ and a reflexive sheaf
  $\sA \subseteq \Sym^{[m]}_{\cC} \Omega^1_X(\log D)$ of rank one with
  $c_1(\sA).A^{n-1} > 0$, then there exists a number $p < \dim X$ and
  reflexive sheaf $\sB \subset \Sym^{[1]}_{\cC} \Omega^p_X(\log D)$ of rank
  one with $\Div_{\cC}(\sB).A^{n-1} > 0$.
\end{prop}
\begin{proof}
  Let $H \subset X$ be a general hyperplane section, and $\gamma: Y \to X$ an
  adapted cover with extra branching along $H$ and cyclic Galois group $G$, as
  in Proposition~\ref{prop:existencesofcovers} on
  page~\pageref{prop:existencesofcovers}. We use Notation~\ref{not:adptweb}
  throughout the proof. Further, let 
  $H_{1,Y}, \ldots, H_{n-1,Y} \in |\gamma^*A|$ be general elements, and
  consider the associated complete intersection curve 
  $$
  C_Y := H_{1,Y} \cap \cdots \cap H_{n-1,Y}.
  $$
  Since Proposition~\ref{prop:HNF} remains invariant if we replace $A$ with a
  positive multiple, we may assume without loss of generality that the
  Mehta-Ramanathan theorem \cite[Thm.~1.2]{Flenner84} holds for $C_Y$,
  i.e.~that taking the Harder-Narasimhan filtration of the sheaf
  $\Omega^{[1]}_{Y}(\log D_{\gamma})_{\rm adpt}$ of adapted differentials
  commutes with restriction to $C_Y$.

  Recall from Remark~\ref{rem:c1ofadptdiffs} that
  $$
  c_1 \left( \Omega^{[1]}_{Y}(\log D_{\gamma})_{\rm adpt} \right) = c_1 \bigl(
  \gamma^*(K_X+D) \bigr).
  $$
  In particular, we have that $c_1 \left( \Omega^{[1]}_{Y}(\log
    D_{\gamma})_{\rm adpt} \right).C_Y = \gamma^*\bigl((K_X+D).A^{n-1}\bigr)
  \leq 0$. On the other hand, it follows immediately from the definition of
  $\cC$-differentials that there exists an inclusion
  $$
  \gamma^{[*]}(\sA) \into \Sym^{[m]} \Omega^{[1]}_{Y}(\log D_{\gamma})_{\rm
    adpt}.
  $$
  By assumption, we have that $c_1\bigl(\gamma^{[*]}(\sA)\bigr).C_Y = \gamma^*
  \bigl( c_1(\sA).A^{n-1} \bigr) > 0$. In particular, it follows that the
  vector bundle $\Omega^{[1]}_{Y}(\log D_{\gamma})_{\rm adpt}|_{C_Y}$ has
  negative degree, but is not anti-nef. Thus, the maximally destabilizing
  subsheaf $\sC_Y \subset \Omega^{[1]}_{Y}(\log D_{\gamma})_{\rm adpt}$ has
  positive slope, $c_1(\sC_Y).C_Y > 0$. It follows that $p := \rank
  \sC_Y < \dim Y = \dim X$, and that $\sB_Y := \det \sC_Y$ is a reflexive
  subsheaf $\sB_Y \subset \Omega^{[p]}_{Y}(\log D_{\gamma})_{\rm adpt}$ of
  rank one and positive slope.

  As a next step, we will construct a sheaf $\sB \subset \Sym^{[1]}_{\cC}
  \Omega^p_X(\log D)$ on $X$. To this end, observe that the line bundle
  $\sO_Y(\gamma^*A)$ is invariant under the action of the cyclic Galois group
  $G$ on the Picard group. Since the sheaf $\Omega^{[p]}_{Y}(\log
  D_{\gamma})_{\rm adpt}$ is also stable under the action of $G$, it follows
  immediately from the uniqueness of the maximally destabilizing sheaf that
  $\sC_Y$ and $\sB_Y$ are likewise $G$-stable. If we set $X^\circ := X_{\reg}
  \setminus \supp(D)$, then
  $$
  \Omega^{[1]}_{Y}(\log D_{\gamma})_{\rm adpt}|_{\gamma^{-1}(X^\circ)} =
  \gamma^{[*]}\bigl(\Sym^{[1]}_{\cC}\Omega^1_{Y}(\log D) |_{X^\circ}\bigr).
  $$
  Using the $G$-invariance of $\sB_Y$ we obtain a sheaf on $X^\circ$, say
  $\sB^\circ \subset \Sym^{[1]}_{\cC}\Omega^p_{X}(\log D) |_{X^\circ}$, such
  that $\gamma^{[*]}(\sB^\circ) = \sB_Y|_{\gamma^{-1}(X^\circ)}$. Let $\sB$ be
  the maximal extension\footnote{We refer to \cite[I.9.4]{EGA1} for a general
    discussion of the maximal extension, or \emph{prolongement canonique} of
    subsheaves.} of $\sB^\circ$ in $\Sym^{[1]}_{\cC}\Omega^p_{X}(\log D)$,
  i.e., the kernel of the natural map
  $$
  \Sym^{[1]}_{\cC}\Omega^p_{X}(\log D) \to \factor
  \Sym^{[1]}_{\cC}\Omega^p_{X}(\log D)|_{\gamma^{-1}(X^\circ)}.\sB^\circ..
  $$
  It is then clear that $\sB$ is reflexive of rank one. In particular, $\sB$
  is locally free wherever $X$ is smooth.

  It remains to show that $\Div_{\cC}(\sB).A^{n-1} > 0$. To this end, recall
  from Remark~\ref{rem:charDefect} that there is an inclusion
  \begin{equation}\label{eq:texs}
    \bigl( \gamma^*(\sB) \otimes \sO_Y(\gamma^* R_{\sB}) \bigr)^{**}
    \into \Omega^{[p]}_Y(\log D_{\gamma})
  \end{equation}
  whose cokernel is torsion free in codimension one. Since the left hand side
  of~\eqref{eq:texs} agrees with $\sB_Y$ generically, reflexivity then implies
  that
  \begin{equation}\label{eq:text}
    \sB_Y \cong \bigl( \gamma^*(\sB) \otimes \sO_Y(\gamma^* R_{\sB})
    \bigr)^{**}.
  \end{equation}
  We observe that the sheaf $\sB_Y$ is locally free along the general curve
  $C_Y$ because the construction of $\sB_Y$ does not depend on the choice of
  $C_Y$. The Isomorphism~\eqref{eq:text} then implies the following:
  \begin{align*}
    \gamma^*\bigl(\Div_{\cC}(\sB).A^{n-1} \bigr) & = \gamma^* \bigl( (c_1(\sB)+c_1(R_{\sB})).A^{n-1} \bigr) && \text{Def.~of $\Div_{\cC}$}\\
    &= c_1(\sB_Y).\bigl( \gamma^*(A) \bigr)^{n-1} & & \text{Isom.~\eqref{eq:text}}\\
    &= c_1(\sB_Y).C_Y > 0. && \text{Choice of $C_Y$}
  \end{align*}
  It follows that $\Div_{\cC}(\sB).A^{n-1} > 0$, as claimed.
\end{proof}

\section{Bogomolov-Sommese vanishing for $\cC$-pairs}
\label{sec:BSv}

In this section we generalize the classical Bogomolov-Sommese Vanishing
Theorem~\ref{thm:classBSv} to sheaves of $\cC$-differentials on $\cC$-pairs
with log canonical singularities.  To do so, we must restrict ourselves to the
case where $X$ is a projective, $\bQ$-factorial, and $\dim X \leq 3$.  The
restriction on the dimension is necessary to apply the Bogomolov-Sommese
vanishing theorem for log canonical threefold pairs\footnote{Building on the
  results of this paper, a stronger version of the Bogomolov-Sommese vanishing
  theorem has meanwhile been shown for $\cC$-pairs of arbitrary dimension,
  \cite[Sect.~7]{GKKP10}.}, \cite[Thm.~1.4]{GKK08}.

\begin{prop}[Bogomolov-Sommese vanishing for 3-dimensional $\cC$-pairs]\label{prop:BSvan}
  Let $(X, D)$ be a $\cC$-pair, as in Definition~\ref{def:orbifold}.  Assume
  that $X$ is projective and $\bQ$-factorial, that $\dim X \leq 3$ and that
  the pair $(X,D)$ is log canonical. If $1 \leq p \leq \dim X$ is any number
  and if $\sA \subseteq \Sym^{[1]}_{\cC} \Omega^p_X(\log D)$ is a reflexive sheaf
  of rank one, then $\kappa_{\cC}(\sA) \leq p$.
\end{prop}
\begin{proof}
  Let $\sA \subseteq \Sym^{[1]}_{\cC} \Omega^p_X(\log D)$ be any given
  reflexive sheaf of rank one. In order to show that $\kappa_{\cC}(\sA) \leq
  p$, let $H \subset X$ be a general hyperplane section, and let $\gamma: Y
  \to X$ be an adapted cover with extra branching along $H$ and cyclic Galois
  group $G$, as in Proposition~\ref{prop:existencesofcovers} on
  page~\pageref{prop:existencesofcovers}.

  As a first step, we show that the pair $(Y, D_{\gamma})$ is log
  canonical. Since $H$ is general, \cite[5.17]{KM98} implies that
  $$
  \discrep(X, D+H) = \min \{ 0, \discrep(X,D)\}.
  $$
  Since $(X,D)$ is log canonical, $\discrep(X,D+H) \geq -1$, so $(X, D+H)$ is
  also log canonical. By \cite[2.27]{KM98}, the pair $\bigl( X, D+
  \frac{N-1}{N} H \bigr)$ is then log canonical as well, where $N$ is the
  least common multiple of those $\cC$-multiplicities that are not infinity,
  as in Proposition~\ref{prop:existencesofcovers}.  Next, recall from
  Lemma~\ref{lem:adjunction} that the log canonical divisor of $(Y,
  D_{\gamma})$ is expressed as follows,
  $$
  K_Y + D_{\gamma} = \gamma^*(K_X + D) + (N-1)\cdot H_{\gamma} = \gamma^*
  \left( K_X + D+ \frac{N-1}{N}H\right).
  $$
  Then since $(X, D+ \frac {N-1} N H)$ is log canonical, so is $(Y,
  D_{\gamma})$, \cite[5.20]{KM98}.

  As a next step, recall from Remark~\ref{rem:pbddbi} that the pull-back
  $\gamma^*(R_{\sA})$ of the defect divisor is an integral divisor on $Y$, and
  consider the sheaf
  $$
  \sB:= \bigl( \gamma^*(\sA) \otimes \sO_Y(\gamma^* R_{\sA}) \bigr)^{**}.
  $$
  We have seen in Corollary~\ref{cor:realizationofkappa} that $\kappa(\sB)
  \geq \kappa_{\cC}(\sA)$, and that there exists an inclusion
  $$
  \sB \into \Sym^{[1]} \Omega_Y^p(\log D_{\ga}) = \Omega_Y^{[p]}(\log
  D_{\ga}).
  $$
  If we show that $\sB$ is $\bQ$-Cartier, then the Bogomolov-Sommese Vanishing
  Theorem for log canonical threefold pairs, \cite[Thm.~1.4]{GKK08}, applies
  to show that $\kappa(\sB) \leq p$. This will yield the claim.  To show that
  $\sB$ is $\bQ$-Cartier, recall that $X$ is $\bQ$-factorial. Since $X$ is
  normal, and $\sA$ is reflexive of rank one, there exists a divisor $D$ on
  $X$ such that $\sA \cong \sO_X(D)$. It follows that
  $$
  \sB \cong \sO_Y\bigl(\gamma^*(D + R_{\sA})\bigr).
  $$
  Since a suitable multiple of the $\bQ$-divisor $D + R_{\sA}$ is Cartier, it
  follows that $\sB$ is $\bQ$-Cartier, as claimed. This ends the proof.
\end{proof}

Combining Propositions~\ref{prop:HNF} and \ref{prop:BSvan}, we obtain a useful
criterion that can be used to show that $\bQ$-Fano $\cC$-pairs $(X, D)$ with
ample anticanonical class $-(K_X+D)$ have Picard number $\rho(X) > 1$. This
will be an essential ingredient in the proof of Theorem~\ref{thm:main}.

\begin{cor}\label{cor:rhoX}
  Let $(X, D)$ be a $\cC$-pair, as in Definition~\ref{def:orbifold}.  Assume
  that $X$ is projective and $\bQ$-factorial, that $\dim X = n \leq 3$ and that
  the pair $(X,D)$ is log canonical.  Let $A$ be an ample Cartier divisor.  If
  $(K_X + D).A^{n-1} \leq 0$ and if there exists a number $m$ and a reflexive
  sheaf $\sA \subseteq \Sym^{[m]}_{\cC} \Omega^1_X(\log D)$ of rank one with
  $c_1(\sA).A^{n-1} > 0$, then $\rho(X)>1$.
\end{cor}
\begin{proof}
  Suppose to the contrary that $\rho(X)=1.$ Given $\sA \subseteq
  \Sym^{[m]}_{\cC} \Omega^1_X(\log D)$ of rank one with $c_1(\sA).A^{n-1} >
  0$, let $\sB \subset \Sym^{[1]}_{\cC} \Omega^p_X(\log D)$ be the reflexive
  rank one sheaf constructed in Proposition~\ref{prop:HNF}, where $p < n$.
  The assumptions that $\rho(X)=1$ and $X$ is $\bQ$-factorial imply that $\sB$
  is $\bQ$-Cartier and a $\bQ$-ample sheaf of $p$-forms.  In particular, by
  Corollary \ref{cor:kappaofample}, $\kappa_{\cC}(\sB) = n$.  But by
  Proposition~\ref{prop:BSvan}, we know that $\kappa_{\cC}(\sB) \leq p < n$, a
  contradiction. It follows that $\rho(X)>1$.
\end{proof}

\part{PROOF OF CAMPANA'S CONJECTURE IN DIMENSION 3}\label{P3}

\section{Setup for the proof of Theorem~\ref*{thm:main}}
\label{sec:setup}

We prove Theorem~\ref{thm:main} in the remainder of the paper. The following
assumptions are maintained throughout the proof.

\begin{assumption}\label{ass:mainproof}
  Let $f^\circ: X^\circ \to Y^\circ$ be a smooth projective family of
  canonically polarized manifolds of relative dimension $n$, over a smooth
  quasi-projective base of dimension $\dim Y^\circ \leq 3$. We assume that the
  family is not isotrivial, $\Var(f^\circ) > 0$, and let $\mu: Y^\circ \to
  \mathfrak{M}$ be the associated map to the coarse moduli space, whose
  existence is shown, e.g. in \cite[Thm.~1.11]{V95}. Arguing by contradiction,
  we assume that $Y^\circ$ is a special variety.
\end{assumption}

\begin{rem}
  Since $Y^\circ$ is special, it is not of log-general type. By
  \cite[Thm.~1.1]{KK08c}, this already implies that the variation of $f^\circ$
  cannot be maximal, i.e., $\Var(f^\circ) < \dim Y^\circ$.
\end{rem}

We also fix a smooth projective compactification $Y$ of $Y^\circ$ such that $D
:= Y \setminus Y^\circ$ is a divisor with simple normal crossings.
Furthermore, we fix a compactification $\overline{\mathfrak{M}}$ of
$\mathfrak{M}$ and let $\mu^{(0)}: Y \dasharrow \overline{\mathfrak{M}}$ be
the associated rational map.

\section{Viehweg-Zuo sheaves on $(Y, D)$}
\label{sec:VZ}

\subsection{Existence of differentials coming from the moduli space}

Under the assumptions spelled out in Section~\ref{sec:setup}, Viehweg and Zuo
have shown in \cite[Thm.~1.4(i)]{VZ02} that $Y^\circ$ carries many logarithmic
pluri-differentials. More precisely, they prove the fundamental result that
there exists a number $m > 0$ and an invertible sheaf $\sA \subseteq \Sym^m
\Omega^1_Y (\log D)$ whose Kodaira-Iitaka dimension is at least the variation
of the family, $\kappa(\sA) \geq \Var(f^\circ)$.

We recall a refinement of Viehweg and Zuo's theorem which asserts that the
``Viehweg-Zuo sheaf'' $\sA$ really comes from the coarse moduli space
$\mathfrak M$. To formulate this result precisely, we use the following
notation.

\begin{notation}\label{not:introB}
  Consider the subsheaf $\sB \subseteq \Omega^1_Y(\log D)$, defined on
  presheaf level as follows: if $U \subset Y$ is any open set and $\sigma \in
  \Gamma\bigl(U,\, \Omega^1_Y (\log D) \bigr)$ any section, then $\sigma \in
  \Gamma\bigl(U,\, \sB \bigr)$ if and only if the restriction $\sigma|_{U'}$
  is in the image of the differential map $d\mu|_{U'} : \mu^* \bigl(
  \Omega^1_{\mathfrak M}\bigr)|_{U'} \to \Omega^1_{U'}$, where $U' \subseteq
  U\cap Y^\circ$ is the open subset where the moduli map $\mu$ has maximal
  rank.
\end{notation}
\begin{rem}
  By construction, it is clear that the sheaf $\sB$ is a saturated subsheaf of
  $\Omega^1_Y (\log D)$. We say that $\sB$ is the saturation of $\Image(d\mu)$
  in $\Omega^1_Y(\log D)$.
\end{rem}

The refinement of Viehweg-Zuo's result is then formulated as follows.

\begin{thm}[\protect{Existence of differentials coming from the moduli space, \cite[Thm.~1.5]{JK09}}]\label{thm:VZimproved} 
  There exists a number $m > 0$ and an invertible subsheaf $\sA \subseteq
  \Sym^m \sB$ whose Kodaira-Iitaka dimension is at least the variation of the
  family, $\kappa(\sA) \geq \Var(f^\circ)$.\qed
\end{thm}

\subsection{Pushing down Viehweg-Zuo sheaves}

In the course of the proof, we will often need to compare Viehweg-Zuo sheaves
on different birational models of a given pair.  The following elementary
lemma shows that Viehweg-Zuo sheaves can be pushed down to minimal models, and
that the Kodaira-Iitaka dimension does not decrease in the process.

\begin{lem}\label{VZbiratl}
  Let $(Z,\Delta)$ be a $\cC$-pair and $\sA \subseteq \Sym_{\cC}^{[m]}
  \Omega_Z^p(\log \Delta)$ a reflexive rank one sheaf for some $m$, $p >0$.
  Let $\lambda: Z \dashrightarrow Z'$ be a birational map whose inverse image
  does not contract any divisor.  If $Z'$ is normal and $\Delta'$ is the
  cycle-theoretic image of $\Delta$, then there exists a reflexive rank one
  sheaf $\sA' \subseteq \Sym_{\cC}^{[m]} \Omega_{Z'}^p(\log \Delta')$ with
  $\kappa_{\cC}(\sA') \geq \kappa_{\cC}(\sA).$
\end{lem}
\begin{subrem}
  Since $\lambda$ is birational, it is clear that any number which appears as
  a coefficient in the divisor $\Delta'$, also appears as a coefficient in
  $\Delta$. Consequently, $(Z', \Delta')$ is again a $\cC$-pair.
\end{subrem}
\begin{proof}[Proof of Lemma~\ref{VZbiratl}]
  The assumption that $\lambda^{-1}$ does not contract any divisor and the
  normality of $Z'$ guarantee that $\lambda^{-1}: Z' \dasharrow Z$ is a
  well-defined embedding over an open subset $U \subseteq Z'$ whose complement
  $\overline Z' := Z' \setminus U$ has codimension $\codim_{Z'} \overline Z'
  \geq 2$, cf.~Zariski's main theorem \cite[V~5.2]{Ha77}.  In particular,
  $\Delta'|_U = \bigl( \lambda^{-1}|_U \bigr)^{-1}(\Delta)$. Let $\iota: U
  \into Z'$ denote the inclusion and set $\sA' := \iota_* \bigl(
  (\lambda^{-1}|_U)^* \sA \bigr)$. Since $\codim_{Z'} \overline Z' \geq 2$,
  the sheaf $\sA'$ is reflexive and agrees with $\sA$ on the open set where
  $\lambda^{-1}$ is an isomorphism. By reflexivity, we obtain an inclusion of
  sheaves, $\sA' \subseteq \Sym_{\cC}^{[m]} \Omega^1_{Z'}(\log
  \Delta')$. Likewise, we obtain that $\Sym_{\cC}^{[d]}\sA' \cong \iota_*
  \bigl( (\lambda^{-1}|_U)^* \Sym_{\cC}^{[d]} \sA \bigr)$ for all $d >
  0$. This gives $h^0\bigl(Z',\, \Sym_{\cC}^{[d]} \sA'\bigr) \geq
  h^0\bigl(Z,\, \Sym_{\cC}^{[d]} \sA \bigr)$ for all $d$, hence
  $\kappa_{\cC}(\sA') \geq \kappa_{\cC}(\sA)$.
\end{proof}

As an immediate corollary, we get that the property of being special is
inherited by preimages of birational morphisms of pairs.

\begin{cor}\label{cor:specialbiratl}
  Let $(Z, \Delta)$ be a $\cC$-pair, and let $\lambda: Z \dasharrow Z'$ be a
  birational morphism whose inverse does not contract a divisor. Assume that
  $Z'$ is normal, and let $\Delta'$ be the cycle-theoretic image of $\Delta$.
  If the $\cC$-pair $(Z', \Delta')$ is special in the sense of
  Definition~\ref{def:specialingen}, and if $E \subset Z$ is any
  $\lambda$-exceptional effective $\mathbb Q$-divisor such that $(Z,
  \Delta+E)$ is a $\cC$-pair, then $(Z, \Delta + E)$ is also special.
\end{cor}
\begin{proof}
  Let $\sA \subseteq \Sym_{\cC}^{[1]} \Omega_Z^p(\log \Delta +E)$ be a
  reflexive rank one sheaf for some $p >0$.  Then by Lemma \ref{VZbiratl},
  there exists a reflexive rank one sheaf $\sA' \subseteq \Sym_{\cC}^{[m]}
  \Omega_{Z'}^p(\log \Delta')$ with $\kappa_{\cC}(\sA') \geq
  \kappa_{\cC}(\sA)$. But since $(Z', \Delta')$ is special, we have that $p >
  \kappa_{\cC}(\sA') \geq \kappa_{\cC}(\sA)$.
\end{proof}

\section{Simplification: factorization of the moduli map}
\label{sec:reduction}

In order to simplify the setup of the proof, we aim to replace the pair $(Y,
D)$ with a pair that is somewhat easier to manage. To this end, we will now
construct a commutative diagram of morphisms between normal varieties,
$$
\xymatrix{ Y \ar@{-->}[d]_{\mu^{(0)}}^{\txt{\scriptsize moduli\\\scriptsize
      map}} && Y^{(1)} \ar[ll]_{\alpha_1} \ar[d]_{\mu^{(1)}}^{\txt{\scriptsize
      conn.\\\scriptsize fibers}} && Y^{(2)} \ar[ll]_{\alpha_2}
  \ar[d]_{\mu^{(2)}}^{\txt{\scriptsize equidim.\\\scriptsize fibers}} &&
  Y^{(3)} \ar[ll]_{\alpha_3} \ar[d]_{\mu^{(3)}}^{\txt{\scriptsize
      equidim.\\\scriptsize fibers}} &&
  Y^{(4)} \ar[ll]_{\alpha_4} \ar@/^0.75pc/[lld]^{\mu^{(4)}} \\
  \overline{\mathfrak M} && Z^{(1)} \ar[ll]^{\beta_1} && Z^{(2)}
  \ar[ll]^{\beta_2} && Z^{(3)}
  \ar[ll]_{\beta_3}^{\txt{\scriptsize discr.~locus\\\scriptsize becomes snc}}, \\
}
$$
where $\beta_2$, $\beta_3$ and all $\alpha_{i}$ are birational morphisms, and
where $Z^{(3)}$ and $Y^{(4)}$ are smooth.

\subsection{Construction of $\boldsymbol{\mu^{(1)}}$ and $\boldsymbol{Z^{(1)}}$}

If necessary, blow up $Y$ outside of $Y^\circ$, in order to obtain a variety
$Y^{(1)}$ which is smooth and where the associated map $Y^{(1)} \dasharrow
\overline{\mathfrak M}$ becomes a morphism. The factorization via a normal
space $Z^{(1)}$ is then obtained by Stein factorization.

\subsection{Construction of $\boldsymbol{Z^{(2)}}$ and $\boldsymbol{Y^{(2)}}$ }

The map $\mu^{(1)}$ induces a natural, generically injective map from
$Z^{(1)}$ into the Chow variety of $Y^{(1)}$,
$$
\gamma: Z^{(1)} \dasharrow \Chow\bigl(Y^{(1)}\bigr), \quad z \mapsto
(\mu^{(1)})^{-1}(z).
$$
Consider a blow-up $\beta_2 : Z^{(2)} \to Z^{(1)}$ such that the composition
$\gamma \circ \beta_2: Z^{(2)} \dasharrow \Chow\bigl(Y^{(1)}\bigr)$ becomes a
morphism and such that $Z^{(2)}$ is smooth. Let $Y^{(2)}$ be the normalization
of the pull-back of the universal family over
$\Chow\bigl(Y^{(1)}\bigr)$. Since the normalization morphism is finite, the
fiber dimension does not change, and the resulting map $\mu^{(2)}$ will have
connected fibers, all of pure dimension $\dim Y^{(2)}-\dim Z^{(2)}$.

\subsection{Construction of $\boldsymbol{Z^{(3)}}$ and $\boldsymbol{Y^{(3)}}$}

Set $D^{(2)} := \supp \bigl(\alpha_1^{-1}\circ \alpha_2^{-1}(D) \bigr)$.
Decompose $D^{(2)}$ into ``horizontal'' components that dominate $Z^{(2)}$ and
``vertical'' components that do not,
$$
D^{(2)} := D^{(2,h)} \cup D^{(2,v)},
$$
and set $D_Z := \mu^{(2)}(D^{(2,v)})$. Further, let $\Delta^{(2)} \subset
Z^{(2)}$ be the discriminant locus of $\mu^{(2)}$. Recall from
Notation~\ref{not:logdiscr} that this is the smallest closed subset such that
$\mu^{(2)}$ is smooth over $Z^{(2)} \setminus \Delta^{(2)}$, and such that the
scheme-theoretic intersection $D^{(2)}\cap (\mu^{(2)})^{-1}(z)$ is a proper
snc divisor in the fiber $(\mu^{(2)})^{-1}(z)$, for all $z \in Z^{(2)}
\setminus \Delta^{(2)}$. Let $\beta_3 : Z^{(3)} \to Z^{(2)}$ be a blow-up such
that $Z^{(3)}$ is smooth, and the preimages $\beta_3^{-1}(\Delta^{(2)})$,
$\beta_3^{-1}(D_Z)$ and $\beta_3^{-1}(\Delta^{(2)} \cup D_Z)$ are all divisors
with snc support. Let $Y^{(3)}$ be the normalization of $Y^{(2)}
\times_{Z^{(2)}} Z^{(3)}$. The induced morphism $\mu^{(3)}$ will again have
connected, equidimensional fibers of pure dimension.  Finally, set
$\Delta^{(3)} := \supp \beta_3^{-1}(\Delta^{(2)} \cup D_Z)$.

\subsection{Construction of $\boldsymbol{Y^{(4)}}$}

Set $D^{(3)} := \supp (\alpha_1 \circ \alpha_2 \circ \alpha_3 )^{-1}(D)$, let
$\alpha_4 : Y^{(4)} \to Y^{(3)}$ be a log resolution of the pair
$\bigl(Y^{(3)}, D^{(3)} \bigr)$ and set $D^{(4)} := \alpha_4^{-1} \bigl(
D^{(3)} \bigr)$.

\subsection{Extension of the boundary}

If $\alpha := \alpha_1 \circ \alpha_3 \circ \alpha_3 \circ \alpha_4$, we
obtain a birational morphism $\alpha : Y^{(4)} \to Y$ where $D =
\alpha_*(D^{(4)})$. The obvious fiber product yields a family of canonically
polarized varieties over $Y^{(4)} \setminus D^{(4)}$ such that $\mu^{(4)}$
factors the moduli map, and such that the associated map $Z^{(3)} \to
\overline{\mathfrak M}$ is generically finite.

To simplify the argumentation further and to define a meaningful $\cC$-base of
the fibration $\mu^{(4)}$, we will now extend the boundary $D^{(4)}$
slightly. To this end, let
$$
E^{(4)} \subseteq (\mu^{(4)})^{-1}(\Delta^{(3)})
$$
be the union of the irreducible components $E' \subseteq (\mu^{(4)})^{-1}
(\Delta^{(3)})$ which are $\alpha_4$-exceptional and not contained in
$D^{(4)}$. By definition of log-resolution, the logarithmic pair $(Y^{(4)},
D^{(4)}+E^{(4)} \bigr)$ is snc, and Corollary~\ref{cor:specialbiratl} asserts
that the pair is special.

\begin{rem}\label{rem:wasserstoff}
  Since $\mu^{(3)}$ is equidimensional, any $\alpha_4$-exceptional divisor is
  also $\mu^{(4)}$-exceptional. By construction of $E^{(4)}$, this implies
  that any $\mu^{(4)}$-exceptional divisor is contained in $D^{(4)}+E^{(4)}$.
\end{rem}

\subsection{Summary, Simplification}
\label{ssec:ss}

Replacing $(Y, D)$ by the pair $(Y^{(4)}, D^{(4)}+E^{(4)})$, if necessary, we
can assume without loss of generality for the remainder of the proof that the
following holds.
\begin{enumerate}
\item The moduli map $\mu^\circ: Y^\circ \to \mathfrak M$ extends to a
  morphism $\mu: Y \to \overline{\mathfrak M}$.
\item There exists a morphism $\pi: Y \to Z$ to a smooth variety $Z$ of
  positive dimension which factors the moduli map as follows
  $$
  \xymatrix{ Y \ar[rr]^{\pi}_{\txt{\scriptsize conn.~fibers}}
    \ar@/^1pc/[rrrr]^{\mu} && Z \ar[rr]_{\txt{\scriptsize generically finite}}
    && \overline{\mathfrak M}.  }
  $$
\item\ilabel{sauerstoff} If $E \subset Y$ is a divisor with $\codim_Z \pi(E)
  \geq 2$, then $E \subseteq D$.
\item\ilabel{fourbar} There exists an snc divisor $\Delta_{\red} \subset Z$
  such that for any point $z \in Z \setminus \Delta_{\red}$, the fiber $Y_z :=
  \pi^{-1}(z)$ is smooth, not contained in $D$, and the scheme-theoretic
  intersection $Y_z \cap D$ is a reduced snc divisor in $Y_z$.
\end{enumerate}

\begin{rem}
  Condition~\iref{fourbar} guarantees that the codimension-one part of the
  discriminant locus of $\pi$ is an snc divisor in $Z$. Together with
  Remark~\ref{rem:wasserstoff} or Condition~\iref{sauerstoff}, this implies
  that the morphism $\pi$ satisfies the Assumptions of ~\ref{ass:fibration}, which
  guarantee the existence of a $\cC$-base with good pull-back and push-forward
  properties for $\cC$-differentials. We are therefore free to use the results
  of Sections~\ref{ssec:pbbf} and \ref{ssec:pfbf} in our setting.
\end{rem}

\section{Proof of Theorem~\ref*{thm:main}}
\label{sec:Pf}

Let $(Z, \Delta)$ be the $\cC$-base of the fibration $\pi$, as constructed in
Section~\ref{sec:ofibr}, Construction~\ref{cons:cdefn}, and note that $\dim Z \leq 2.$
By construction, it is
clear that $\supp(\Delta) \subseteq \Delta_{\red}$, where $\Delta_{\red}
\subset Z$ is the divisor introduced in Section~\ref{ssec:ss} above. In
particular, the divisor $\Delta$ has snc support, and the pair $(Z, \Delta)$
is dlt, \cite[Cor.~2.35 and Def.~2.37]{KM98}. Since the logarithmic pair
$(Y,D)$ is special by assumption, Corollary~\ref{cor:baseofspecialisspecial}
implies that $(Z, \Delta)$ is a special $\cC$-pair in the sense of
Definition~\ref{def:specialingen}.

Next, let $\sB \subseteq \Omega^1_Y(\log D)$ be the sheaf introduced in
Notation~\ref{not:introB} above.  By Theorem~\ref{thm:VZimproved} there exists
an invertible, saturated sheaf
$$
\sA \subseteq \Sym^m \sB \subseteq \Sym^m \Omega_Y^1(\log D)
$$
with $\kappa(\sA) \geq \Var(f^\circ) = \dim Z$. Since $Z$ is generically
finite over $\overline{\mathfrak M}$, the sheaf $\sB$ is also the saturation
of the image of
$$
d\pi: \pi^*(\Omega^1_Z) \to \Omega^1_Y(\log D).
$$
Corollary~\ref{cor:pdvzs} thus asserts that $\sA$ descends to a reflexive
subsheaf $\sA_Z \subseteq \Sym_{\cC}^{[m]}\OM_Z^1(\log \Delta)$ of rank one,
with $\kappa_{\cC}(\sA_Z)= \dim Z$.

\subsection{Case: $\boldsymbol{\dim Z = 1}$.}  

Since $Z$ is a curve, $\Sym_{\cC} ^{[m]} \Omega_Z^1(\log \Delta)$ is of rank
one and therefore equals $\sA_Z$. Recall from Remark~\ref{rem:iokuop} that
this asserts that $\kappa_{\cC}\bigl(\Sym_{\cC} ^{[1]} \Omega_Z^1(\log
\Delta)\bigr) =1$, contradicting the fact that the $\cC$-pair $(Z, \Delta)$ is
special. This ends the proof in case $\dim Z = 1$.

\subsection{Case: $\boldsymbol{\dim Z = 2}$.}

Applying the the minimal model program to the dlt pair $(Z, \Delta)$, we
obtain a birational morphism\footnote{Since $Z$ is a surface, the minimal
  model program does not involve flips.} $\lambda: Z \to Z_{\lambda}$. Set
$\Delta_\lambda := \lambda_*(\Delta)$, and recall that $Z_{\lambda}$ is
$\bQ$-factorial, that the pair $(Z_{\lambda}, \Delta_{\lambda})$ is dlt and
that it does not admit divisorial contractions.

Let $\sA_{\lambda} \subset \Sym_{\cC}^{[m]} \OM_{Z_{\lambda}}^1(\log
\Delta_{\lambda})$ be the Viehweg-Zuo sheaf associated to $\sA_Z \subset
\Sym_{\cC} ^{[m]} \OM_{Z}^1(\log \Delta)$, as given by Lemma~\ref{VZbiratl},
and note that $\kappa_{\cC}(\sA_{\lambda})= \dim Z=2$.  For convenience of
argumentation, we consider the possibilities for $\kappa(K_Z + \Delta)$
separately.

\subsubsection{Sub-case: $\kappa(K_Z + \Delta) = -\infty$}
\label{ssec:kappaminfty}

In this case, the pair $(Z_{\lambda}, \Delta_{\lambda})$ is either $\mathbb
Q$-Fano and has Picard number $\rho(Z_{\lambda}) = 1$, or $(Z_{\lambda},
\Delta_{\lambda})$ admits an extremal contraction of fiber type and has the
structure of a proper Mori fiber space.

The case $\rho(Z_{\lambda}) = 1$, however, is ruled out by
Corollary~\ref{cor:rhoX}: if $\rho(Z_{\lambda}) = 1$, then $K_{Z_\lambda} +
\Delta_\lambda$ is anti-ample. If $A \subset Z_{\lambda}$ is a general
hyperplane section, this gives $(K_{Z_\lambda} + \Delta_\lambda).A <
0$. Corollary~\ref{cor:rhoX} then asserts that $\rho(Z_{\lambda}) > 1$,
contrary to our assumption.

We thus obtain that $\rho(Z_{\lambda}) > 1$, and that there exists a
fiber-type contraction $\pi: Z_{\lambda} \to B$, where $B$ is a curve. If $F$
is a general fiber of $\pi$, then $F \iso \P^1$, $F$ is entirely contained in
the snc locus of $(Z_\lambda, \Delta_\lambda)$, and $F$ intersects
$\Delta_\la$ transversely. Since the normal bundle $N_{F/Z_{\lambda}}$ is
trivial and $-(K_F + \Delta_\lambda|_F)$ is nef,
Proposition~\ref{prop:gennbld} asserts that $\Sym_{\cC}^{[m']}
\Omega_{Z_\lambda}^1(\log \Delta_{\lambda})|_F$ is anti-nef, for all numbers
$m' \in \mathbb N^+$. It follows that
$$
\Sym^{[m']}_{\cC} \sA_\lambda|_F \subset \Sym_{\cC}^{[m' \cdot m]}
\Omega_{Z_\lambda}^1(\log \Delta_{\lambda})|_F
$$
is a subsheaf of an anti-nef bundle, hence anti-nef for all $m' \in \mathbb
N^+$. This clearly contradicts $\kappa_{\cC}(\sA_{\lambda})= \dim Z = 2$.

\subsubsection{Sub-case: $\kappa(K_Z + \Delta) = 0$}

In this case, the classical Abundance Theorem \cite[Sect.~3.13]{KM98} asserts
that there exists a number $n \in \mathbb N^+$ such that
\begin{equation}\label{eq:trs}
  \sO_{Z_\lambda}\bigl( n\cdot(K_{Z_{\lambda}} + \Delta_{\lambda})\bigr) \cong
  \sO_{Z_\lambda}.
\end{equation}

If the boundary divisor $\Delta_\lambda$ is empty, then the $\cC$-pair
$(Z_\lambda, \Delta_\lambda)$ is a logarithmic pair for trivial reasons, and
\cite[Prop.~9.1]{KK08c} implies that $\kappa(\sA_\lambda) \leq 0$, a
contradiction.  It follows that $\Delta_\lambda$ is \emph{not} empty.

For sufficiently small $\varepsilon_0 \in \mathbb Q^+$, we can therefore
consider the dlt pair $(Z_{\lambda}, (1-\varepsilon_0)\Delta_{\lambda})$.
Equation~\eqref{eq:trs} implies that $-\bigl(K_{Z_{\lambda}} +
(1-\varepsilon_0)\Delta_{\lambda}\bigr)$ is $\bQ$-effective. In particular, we
have that $\kappa \bigr(K_{Z_{\lambda}} + (1-\varepsilon_0)\Delta_{\lambda}
\bigl) = - \infty$. We can therefore run the minimal model program of the pair
$\bigl(Z_{\lambda}, (1-\varepsilon_0)\Delta_{\lambda} \bigr)$, in order to
obtain a birational morphism $\mu: Z_{\lambda} \to Z_{\mu}$ to a normal,
$\bQ$-factorial variety. Set $\Delta_{\mu} := \mu_*(\Delta_{\lambda})$. As
before, Lemma~\ref{VZbiratl} gives the existence of a Viehweg-Zuo sheaf
$\sA_{\mu} \subseteq \Sym^m_{\cC} \Omega_{Z_\mu}^1(\log \Delta_{\mu})$ with
$\kappa_{\cC} (\sA_{\mu}) = 2$.

To continue, observe that the map $\mu$ is also a minimal model program of the
pair $(Z_{\lambda}, (1-\varepsilon)\Delta_{\lambda})$, for any sufficiently
small number $\varepsilon \in \mathbb Q$. In particular, the pair
$\bigl(Z_{\mu}, (1-\varepsilon) \Delta_{\mu}\bigr)$ is dlt for all
$\varepsilon$, its Kodaira dimension is $\kappa \bigl(K_{Z_{\mu}}+
(1-\varepsilon)\Delta_{\mu} \bigr) = -\infty$, and the pair $(Z_{\mu},
\Delta_{\mu})$ is hence dlc \cite[9.4]{KK08c}, in particular log canonical. In
this setting, the arguments of the previous Section~\ref{ssec:kappaminfty}
apply verbatim.

\subsubsection{Sub-case: $\kappa(K_Z + \Delta) > 0$}

The Abundance Theorem guarantees the existence of a regular Iitaka-fibration
$\pi: Z_{\la} \to B$, such that $K_{Z_\lambda} + \Delta_\lambda$ is trivial on
the general fiber $F$. The same argumentation as in
Section~\ref{ssec:kappaminfty} applies to show that $\Sym^{[m']}_{\cC}
\sA_\lambda$ is anti-nef for all $m' \in \mathbb N^+$, contradicting
$\kappa_{\cC}(\sA_{\lambda})= \dim Z = 2$. This finishes the proof in the case
$\dim Z = 2$ and ends the proof of Theorem~\ref{thm:main}. \qed

\providecommand{\bysame}{\leavevmode\hbox to3em{\hrulefill}\thinspace}
\providecommand{\MR}{\relax\ifhmode\unskip\space\fi MR}
\providecommand{\MRhref}[2]{%
  \href{http://www.ams.org/mathscinet-getitem?mr=#1}{#2}
}
\providecommand{\href}[2]{#2}

\end{document}